\documentclass{article}

\usepackage{amsfonts}
\usepackage{amssymb}

\usepackage[usenames,dvipsnames]{color}
\usepackage{amsmath}
\usepackage{amsfonts}
\usepackage{amssymb}
\usepackage{algorithmic,algorithm}
\usepackage{eurosym}
\usepackage{calc}
\usepackage{hyperref}

\usepackage{graphicx}
\usepackage{caption}
\usepackage{subcaption}
\usepackage{placeins}

\usepackage[margin=1in]{geometry}

\usepackage{amsthm}

\usepackage{booktabs}
\usepackage{multirow}

\newtheorem{theorem}{Theorem}
\newtheorem{lemma}{Lemma}

\newtheorem{example}{Example}
\newtheorem{definition}{Definition}
\newtheorem{assumption}{Assumption}
\newtheorem{corollary}{Corollary}

\newcommand{\R}{\mathbb{R}}
\newcommand{\X}{\mathbb{X}}
\newcommand{\U}{\mathbb{U}}
\newcommand{\TTS}{\text{TTS}}

\newcommand{\G}{\mathcal{G}}
\newcommand{\E}{\mathcal{E}}
\newcommand{\V}{\mathcal{V}}
\newcommand{\M}{\mathcal{M}}
\newcommand{\D}{\mathcal{D}}
\newcommand{\N}{\mathcal{N}}

\newcommand{\de}{d_e \big( \rho_e(t) \big)}
\newcommand{\se}{s_e \big( \rho_e(t) \big)}

\title{An Exact Convex Relaxation of the Freeway Network Control Problem with Controlled Merging Junctions}% <-this % stops a spac

\author{Marius Schmitt\thanks{Automatic Control Laboratory, ETH Zurich, Switzerland,     {\tt\small <schmittm,lygeros>@control.ee.ethz.ch}}
\thanks{Corresponding author} \ and John Lygeros$^*$% <-this % stops a space
}

\begin{document}

\maketitle

\begin{abstract}
We consider the freeway network control problem where the aim is to optimize the operation of traffic networks modeled by the Cell Transmission Model via ramp metering and partial mainline demand control. Optimal control problems using the Cell Transmission Model are usually non-convex, due to the nonlinear fundamental diagram, but a convex relaxation in which demand and supply constraints are relaxed is often used. Previous works have established conditions under which solutions of the relaxation can be made feasible with respect to the original constraints. In this work, we generalize these conditions and show that the control of flows into merging junctions is sufficient to do so if the objective is to minimize the total time spent in traffic. We derive this result by introducing an alternative system representation. In the new representation, the system dynamics are concave and state-monotone. We show that exactness of the convex relaxation of finite horizon optimal control problems follows from these properties. Deriving the main result via a characterization of the system dynamics allows one to treat arbitrary monotone, concave fundamental diagrams and several types of control for merging junctions in a uniform manner. The derivation also suggests a straightforward method to verify if the results continue to hold for extensions or modifications of the models studied in this work.
\end{abstract}

\section{Introduction} \label{sec:introduction} %%% SECTION ====================================================== %

We study the freeway network control (FNC) problem, that is, the problem of optimal operation of freeway traffic for networks modeled using a variant of the cell transmission model (CTM) \cite{daganzo1994cell,daganzo1995cell}, with the standard objective of minimizing the total time spent (TTS). The CTM is a first-order traffic model obtained as a discretization of the kinematic wave model \cite{lighthill1955kinematic,richards1956shock}. It describes road traffic by a linear conservation law and the nonlinear fundamental diagram, which models the relationship between traffic flow and traffic density. Finite-horizon optimal control problems for systems modeled by the CTM lead to nonconvex optimization problems in general, due to the nonlinear fundamental diagram. However, a convex optimization problem is obtained if the demand and supply constraints encoded in the fundamental diagram are relaxed. In particular, a linear program (LP) is obtained if a triangular or trapezoidal fundamental diagram is used \cite{ziliaskopoulos2000linear}. In general, an optimal solution of the relaxed problem does not satisfy the dynamics of the CTM, but subsequent work has identified conditions for which the relaxation yields solutions to the original problem. In particular, it turns out that for a freeway segment with only onramp junctions and off-ramp junctions solutions to the relaxed problem are feasible with respect to the CTM dynamics \cite{gomes2006optimal}. This result relies on the assumption that onramps are metered and inflow from onramps is not obstructed by mainline congestion, while off-ramps are assumed to be uncongested and hence, they do not obstruct mainline flow via congestion spillback. In \cite{gomes2008behavior} it has been shown that the corresponding CTM model is in fact a discrete-time, monotone system, a generalization of monotone maps \cite{hirsch2005monotone} to systems with inputs. Basic definitions and results on monotone systems are presented in \cite{angeli2003monotone} for the analogous continuous-time case.

It is natural to ask whether monotonicity properties can be leveraged to facilitate the analysis and control of systems based on the CTM. However, it turns out that first-in, first-out (FIFO) diverging junctions as used in the CTM are not monotone \cite{coogan2015compartmental}. Alternative models for diverging junctions with monotone dynamics have been suggested \cite{lovisari2014stability,lovisari2014stability2}, but these models do not preserve the turning rates. In \cite{coogan2016stability} it is shown that the CTM dynamics satisfy a mixed-monotonicity property instead, but it is also suggested that the non-monotone dynamics of FIFO diverging junctions are exactly what dynamic traffic control should target in order to realize improvements over the uncontrolled case. Monotonicity has also been used to analyze robustness of optimal trajectories \cite{como2016convexity}. In addition, traffic routing problems have been considered \cite{como2013robust1,como2013robust2,como2015throughput}. In such problems, the turning rates are not fixed a priori, but they are (partially) actuated variables instead. With time-varying turning rates, diverging junctions do not exhibit FIFO dynamics, allowing one to circumvent the issues arising from the non-monotone effects. In particular, monotone routing policies are resilient to non-anticipated capacity reductions in individual links \cite{como2013robust1,como2013robust2}. Subsequently, a class of distributed, monotone routing policies was proposed, that make use of the implicit back-propagation of congestion to stabilize maximal-throughput equilibria \cite{como2015throughput}.

It has also been suggested that solutions to the relaxed FNC problem (using relaxed demand and supply constraints) for arbitrary networks can be made feasible if traffic demand control is available in every cell of the CTM, for example via variable speed limits \cite{muralidharan2012optimal,como2016convexity}. However, it is questionable whether the assumption of demand control in every cell is realistic, in particular for freeway networks. Even if variable speed limits are implemented, possible operation modes are usually restricted, with only few distinct speed limits to chose from and constraints on how often these may change. Therefore, a crucial question is whether demand control in every cell is necessary to achieve the optimal cost of the relaxed problem, or if, for example, ramp metering is sufficient to do so. A partial answer is known for the special case of a symmetric triangular fundamental diagram in which the congestion wave speed is equal to the free-flow velocity in every cell. In this case, the solution to the relaxed FNC problem can be made feasible by using priority control for flows into merging junctions \cite[Proposition 2]{como2016convexity}.

In this work, we generalize these results and consider CTM networks with FIFO diverging junctions and concave (but not necessarily symmetric or even piece-wise affine (PWA)) fundamental diagrams. We show that if the objective is to minimize the TTS, control of merging flows is sufficient to achieve the same cost as in the relaxed problem. This result allows us to use the convex, relaxed problem to efficiently compute solutions of the original nonconvex FNC problem. The main result of this work relies on the analysis of a novel, alternative system representation of the CTM. It turns out that the system dynamics are concave and state-monotone in the new representation. This allows us to employ results originally derived for convex, monotone systems \cite{rantzer2014control} to show equivalence of the convex relaxation to the nonconvex optimal control problem. We generalize existing results, in particular \cite{como2016convexity} where a related problem is addressed, in the following ways:
 \begin{itemize}
	\item Our main result applies to CTM networks with general concave, monotone fundamental diagrams. The existing result holds only for affine demand and supply functions with \emph{identical} slopes (of opposite sign)\footnote{ A different result \cite[Proposition 1]{como2016convexity} allows for more general fundamental diagrams, but it requires demand control in \emph{every} cell, as opposed to only in cells immediately upstream of merging junctions.}, i.e., the case when the free-flow speed is equal to the congestion wave speed. Real-world free-flow speeds are typically significantly larger than congestion wave speeds.
	\item The main result is based on a novel system reformulation, in which the system dynamics are state-monotone and concave. The reformulation of the system dynamics links the result to properties of the dynamical system itself and suggests a straightforward method to verify if the results continue to hold for extensions or modifications of the models studied in this work.
\end{itemize}

This paper is structured as follows: in Section \ref{sec:csm}, we introduce results on the optimal control of concave, state-monotone systems that will be used subsequently. The freeway network model is introduced in Section \ref{sec:networks}. In Section \ref{sec:CCTM}, we perform a transformation of the system equations to derive an equivalent system representation and show that it is concave and state-monotone. This allows us to prove the main result of this work: the derivation of an exact, convex relaxation of the FNC problem for networks with controlled merging junctions. We contrast the dynamics of merging and diverging junctions in the original system model with the alternative representation in Section \ref{sec:closer} in order to demonstrate the applicability and limitations of our results. In Section \ref{sec:application}, we apply the main result to compute optimal open-loop control inputs for two freeway network examples, a real world freeway and an artificial freeway network designed to showcase the behavior of merging and diverging junctions. In Section \ref{sec:conclusions}, we summarize our contributions and provide suggestions for future work. \\

\section{Preliminaries} \label{sec:csm} %%% SECTION =========================================================== %

Consider the discrete-time dynamical system with state $x(t) \in \X  \subseteq \R^n$, input $u(t) \in \U \subseteq \R^m$ and dynamics
\begin{equation*}
x(t+1) = f_t \big( x(t),u(t) \big) ~,
\end{equation*}
with $f_t : \X \times \U \rightarrow \X$ and the index $t$ indicating that the dynamics are allowed to be time-varying. In this work, we are interested in the special case when the system dynamics are concave and state-monotone.
\begin{definition}
\label{def:state_monotone}
A function $f : \X \times \U \rightarrow \R^l$ is called \emph{state-monotone} if for all $x_1 \in \X$, $x_2 \in \X$ such that $x_1 \geq x_2$ it holds that $f(x_1,u) \geq f(x_2,u)$ for all $u \in \U$.
\end{definition}
In this definition as well as in the remainder of this work, all inequalities are interpreted elementwise. 
\begin{definition}
\label{def:csm}
A dynamical system is called \emph{concave, state-monotone} if 
\begin{itemize}
\setlength\itemsep{-0.2em}
	\item[(i)] the system dynamics $f_t( x, u)$ are state-monotone and
	\item[(ii)] the system dynamics $f_t( x, u)$ are jointly concave in $x$ and $u$ for all $t$ and the sets $\X$ and $\U$ are closed and convex.
\end{itemize}
\end{definition}
Our definition of a state-monotone system is closely related to the standard definition of a cooperative system, a special case of an order preserving monotone system \cite{angeli2003monotone}. Note that the previous reference considers continuous-time systems while we consider the discrete-time case exclusively.

Systems whose dynamics are both monotone and convex in state and input have been studied in \cite{rantzer2014control} and much of the remainder of this section follows their reasoning and results. Our motivation to present the results in terms of maximization of the cost of operation of a concave system, instead of minimization of the cost of a convex system, is their subsequent application to the maximization of flows in the FNC problem. Unlike \cite{rantzer2014control}, we drop the assumption of monotonicity in the inputs, but make additional assumptions on the control objective and potential constraints as in \cite{schmitt2017convex}. In particular, we assume that the control objective is the minimization of a stage-wise, potentially time-varying cost $c_t : \X \times \U \to \R$ which is itself concave and state-monotone for every $t$. We also allow (potentially time-varying) input-state constraints $g_t ( x(t) , u(t) ) \geq 0$, where the functions $g_t : \X \times \U \to \R^c$  (with $c$ the number of constraints) are also assumed to be concave and state-monotone. For concave, state-monotone systems equipped with a stage-wise concave, state-monotone cost and concave, state-monotone constraints, we consider the finite-horizon optimal control problem
\begin{equation} %%% EQUATION
\begin{array}{rrll}
\mathcal{P}^* := & \underset{ x(t), u(t) }{\text{maximize}} & \sum_{t = 0}^{T-1} c_t \big( x(t), u(t) \big) + c_T \big( x(T) \big) \\ [2ex]
 & \text{subject to} & x(t+1) = f_t \big( x(t), u(t) \big) &\quad\forall t \in \{0,1,\dots,T-1\} \\[1ex]
 & & g_t \big( x(t), u(t) \big) \geq 0 &\quad\forall t \in \{0,1,\dots,T-1\}  \\[1ex]
 & & x(0) ~\text{given},
\end{array}
\label{eq:finitehorizon}
\end{equation}
Here, $\mathcal{P}^*$ denotes the optimal value. This problem is nonconvex due to the nonlinear equality constraints encoding the system dynamics. However, it turns out that the convex relaxation
\begin{equation} %%% EQUATION
\begin{array}{rrll}
\mathcal{R}^* := & \underset{ z(t), v(t) }{\text{maximize}} & \sum_{t = 0}^{T-1} c_t \big( z(t), v(t) \big) + c_T \big( z(T) \big) \\ [2ex]
 & \text{subject to} & z(t+1) \leq f_t \big( z(t), v(t) \big) &\quad\forall t \in \{0,1,\dots,T-1\} \\[1ex]
 & & g_t ( z(t), v(t) ) \geq 0 &\quad\forall t \in \{0,1,\dots,T-1\} \\[1ex]
 & & z(0) = x(0) ~\text{given}
\end{array}
\label{eq:relaxation}
\end{equation}
with optimal value $\mathcal{R}^*$ can be used to solve the original problem. Note that the equality constraints have been relaxed in problem \eqref{eq:relaxation}.
\begin{theorem}
\label{theorem:csm}
For a concave, state-monotone system $x(t+1) = f_t(x(t),u(t))$ with concave, state-monotone stage costs $c_t(x(t),u(t))$ and concave, state-monotone constraints $g_t( x(t), u(t) )$, the convex relaxation \eqref{eq:relaxation} is exact in the sense that the optimal values $\mathcal{P}^*$ and $\mathcal{R}^*$ coincide and hence, any optimizer of \eqref{eq:finitehorizon} is also an optimizer of \eqref{eq:relaxation}.
\end{theorem}
The relaxation \eqref{eq:relaxation} may have non-unique optimizers, some of which are not feasible in \eqref{eq:finitehorizon}, but the proof also reveals how to construct an optimizer of the original problem from any optimizer of the convex relaxation.
\begin{proof}
The functions describing the objective, the system dynamics and the constraints are all concave and thus the relaxation \eqref{eq:relaxation} is indeed a convex optimization problem. Assume now that $\big( z^*(t), ~ v^*(t) \big)$ is an optimizer of the relaxed problem. Consider the candidate solution $u^*(t) = v^*(t)$, $x^*(0) = z^*(0)$ and $x^*(t+1) = f_t(x^*(t),u^*(t))$. Note that from $x^*(0) = z^*(0)$ and state-monotonicity of the dynamics, it follows inductively that
\begin{equation*}
x^*(t) = f_t \big( x^*(t-1),u^*(t-1) \big) \geq f_t \big( z^*(t-1),u^*(t-1) \big) \geq z^*(t) \quad \forall t \in \{1, \dots, T\} ~.
\end{equation*}
By state-monotonicity of the constraints, this in turn implies that
\begin{equation*}
g_t \big( x^*(t), u^*(t) \big) \geq g_t \big( z^*(t), v^*(t) \big) \geq 0 \quad \forall t \in \{1, \dots, T\} ~,
\end{equation*}
that is, $\big( x^*(t), ~u^*(t) \big)$ is a feasible solution of the original problem. Furthermore, by state-monotonicity of the cost,
\begin{equation*}
\mathcal{P}^* \geq \sum_{t = 0}^{T-1} c_t \big( x^*(t), u^*(t) \big) \geq \sum_{t = 0}^{T-1} c_t \big( z^*(t), v^*(t) \big) = \mathcal{R}^* ~.
\end{equation*}
Since the second problem is a relaxation of the first, $\mathcal{P}^* \leq \mathcal{R}^*$, and the claim follows.
\end{proof}
The existence of an exact, convex relaxation in the sense of Theorem \ref{theorem:csm} allows the efficient solution of finite-horizon optimal control problems and is the main reason for our interest in concave, state-monotone systems. An analogous statement can be made for convex, state-monotone systems if the objective is the minimization of a stage cost.

\section{Problem statement} \label{sec:networks} %%% SECTION ======================================================== %

We consider the freeway network control (FNC) problem \cite{ziliaskopoulos2000linear,muralidharan2012optimal,como2016convexity}, where actuation is restricted to demand control but traffic routing cannot be influenced. Our traffic model is very similar to the one used in the latter reference, although we make additional assumptions, in particular on merging junctions, that will be detailed later. Similar models are also studied e.g.\ in \cite{coogan2015compartmental, coogan2016stability}, where the focus is on stability of system equilibria. These models are based on the CTM \cite{daganzo1994cell,daganzo1995cell}.
\begin{figure}[t] %%% FIGURE
	\centering
	\begin{subfigure}[b]{0.5\textwidth}
		\includegraphics[width=\textwidth]{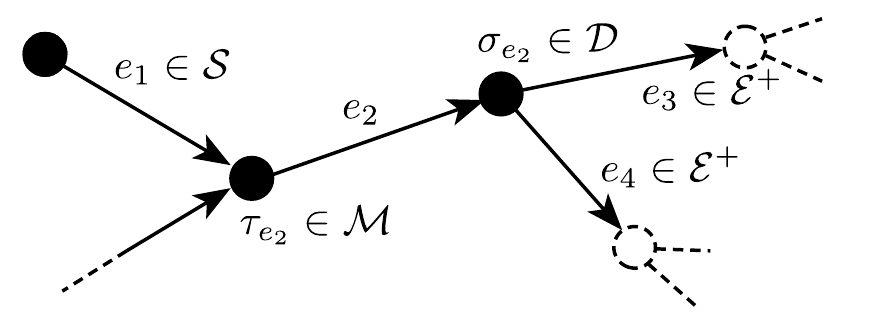}
		\caption{Example notation}
		\label{fig:graph_notation}
	\end{subfigure} \hspace{1cm}	
	\begin{subfigure}[b]{0.27\textwidth}
		\includegraphics[width=\textwidth]{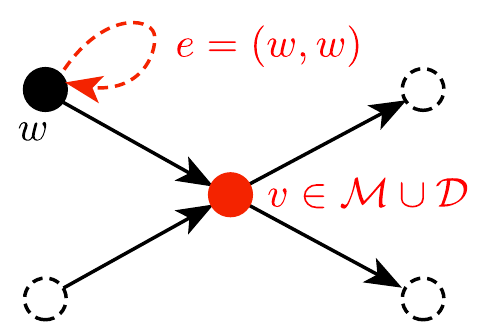}
		\caption{Violations of Assumption \ref{assumption:graph}}
		\label{fig:graph_violation}
	\end{subfigure}
	\caption{The notation is illustrated in the subgraph on the left. The subgraph on the right violates Assumption \ref{assumption:graph}, in particular, edge $e$ violates the assumption of no self loops and vertex $v$ violates the assumption that vertices are not both merging and diverging junctions.}
	\label{fig:graph}
\end{figure} %%% FIGURE
The freeway network is represented by a directed graph $\G = (\V ,\E )$ in which the edges $e \in \E \subseteq \V \times \V$ represent \emph{cells}, that is, sections of road, and the vertices represent (merging or diverging) junctions or interfaces between consecutive cells. We denote the tail of an edge $e$ as $\tau_e$ and the head as $\sigma_e$, i.e., $e = (\tau_e,\sigma_e)$. Traffic flows from tail $\tau_e$ to head $\sigma_e$. For an edge $e$ we define the set of \emph{downstream} cells as $\E^+(e) := \{ i \in \E : \tau_i = \sigma_e \}$. The indegree, that is, the number of incoming edges of a junction (vertex) $v$ is denoted $\deg^-(v)$ and the outdegree, the number of edges leaving a junction $v$, is $\deg^+(v)$. Junctions with two or more outgoing edges are called \emph{diverging} junctions and the set of all diverging junctions is denoted $\mathcal{D} := \{ v \in \V : \deg^+(v) > 1 \}$. Similarly, vertices with more than one incoming edge are called \emph{merging} junctions and we write $\mathcal{M} := \{ v \in \V : \deg^-(v) > 1 \}$ for the set of all merging junctions. We use $\N := \{e \in \E : \sigma_e \in \M \}$ to denote the set of all cells directly upstream of a merging junction. A vertex with $\deg^+(v) = 0$ is called a \emph{sink}. An edge without upstream edges is called a \emph{source cell} and we introduce the set of all source cells $\mathcal{S} := \{ e \in \E : \deg^-( \tau_e ) = 0 \}$. The notation is summarized in Figure \ref{fig:graph}.

\begin{assumption} \label{assumption:graph}
The directed network graph $\G$ does not contain self-loops, that is, edges of the form $e = (v,v)$. In addition, we assume that no vertex is both a merging and diverging junction ($\M ~\cap~ \D = \varnothing$) and that merging junctions are not sinks.
\end{assumption}

The state of each cell is described by the traffic \textit{density} $\rho_e(t)$, i.e.,\ the number of cars in a cell divided by the \emph{length} of the cell $l_e$. We adopt a discrete-time model in which the evolution of the system is described by flows of cars during discrete time intervals of duration $\Delta t$. For each cell, we define the \emph{flow} $\phi_e(t)$ as the traffic flow out of cell $e$ during the time interval $t$. We introduce \emph{turning rates} $0 < \beta_{i,j} \leq 1$ for any two adjacent cells $j$ and $i$ ($\sigma_j = \tau_i$) to model how traffic flows are distributed onto multiple downstream cells. The turning rate $\beta_{i,j}$ models the percentage of flow leaving cell $j$ that travels to cell $i$. We assume that the turning rates are invariant in time. To simplify notation, we also define $\beta_{i,j} = 0$ whenever $\sigma_j \neq \tau_i$. The conservation of traffic requires that $ \sum_{j \in \E} \beta_{i,j} \leq 1$. We allow for $ \sum_{j \in \E} \beta_{i,j} < 1$ and we assume that the percentage of flow that is not distributed onto the downstream cells leaves the network, e.g.\ via an off-ramp. In addition to the flows within the network, external inflows $w_e(t)$ may enter the source cells $e \in \mathcal{S}$. The traffic densities evolve according to the conservation law
\begin{equation}
\label{eq:ctm_conservation}
\rho_e(t+1) = \rho_e(t) + \frac{\Delta t}{l_e} \cdot \left( \sum_{i \in \E} \beta_{e,i} \phi_i(t) ~ - \phi_e(t) + w_e(t) \right) \quad \forall e \in \E ~.
\end{equation}
Note that for non-source cells $e \notin \mathcal{S}$, the external inflows are zero ($w_e(t) = 0$), while for source cells, $\sum_{i \in \E} \beta_{e,i} \phi_i(t)  = 0$ according to the definition of $\mathcal{S}$. The CTM is a first-order model where the flows $\phi(t)$ are computed as functions of the states $\rho(t)$. In general, the traffic flows depend on the traffic \emph{demand}, the number of cars that seek to travel downstream within a time interval, and the \emph{supply} of free space in downstream cells. To model this behavior, we introduce demand and supply functions for each cell. 
%These are often combined in the \textit{fundamental diagram} of a cell, as depicted in Figure \ref{fig:fd}. 
In the work of \cite{daganzo1994cell,daganzo1995cell}, piecewise-affine demand and supply functions were derived from the Godunov discretization of the Lighthill-Whitham-Richards (LWR) model \cite{lighthill1955kinematic,richards1956shock}. In practice, one might want to consider more general demand and supply functions, in order to better approximate real-world data, see e.g.\ \cite{como2016convexity, coogan2016stability,lovisari2014stability} for recent examples. In the remainder of this work, we will assume that:
\begin{assumption} %%% ASSUMPTION
\label{assumption:fd}
For every cell $e$ a maximal density $\bar\rho_e$, called the traffic jam density, is defined. The demand functions $d_e: [0,\bar \rho_e] \rightarrow \mathbb{R}_+$ are concave, Lipschitz continuous with Lipschitz constant $\gamma_d$, nondecreasing and satisfy $d_e(0) = 0$. Conversely, the supply functions $s_e: [0,\bar \rho_e] \rightarrow \mathbb{R}_+$ are concave, Lipschitz-continuous with Lipschitz constant $\gamma_s$, nonincreasing and satisfy $s_e(\bar \rho_e) = 0$. Furthermore, the sampling time $\Delta t$ is chosen such that it satisfies the bound
\begin{subequations}
\begin{align}
\Delta t \leq  \frac{ \min_{e \in E} \{l_e \}}{\max \{ \gamma_d, \gamma_s \} } ~.
\end{align}
\end{subequations}
We allow for cells with infinite capacity, that is, cells with both infinite traffic jam densities $\bar\rho_e = +\infty$ and infinite supply $s_e( \rho_e(t) ) = +\infty$ for all $\rho_e(t) \in [0, +\infty)$. In such a case, the demand function $d_e( \rho_e(t) )$ has to be defined for $\rho_e(t) \in [0, +\infty)$.
\end{assumption} % end assumption

\begin{figure} %%% FIGURE
	\centering
	\begin{subfigure}[b]{0.3\textwidth}
	  	\setlength{\unitlength}{0.1\textwidth}
  		\begin{picture}(10,9)
    			\put(0,0){\includegraphics[width=5cm]{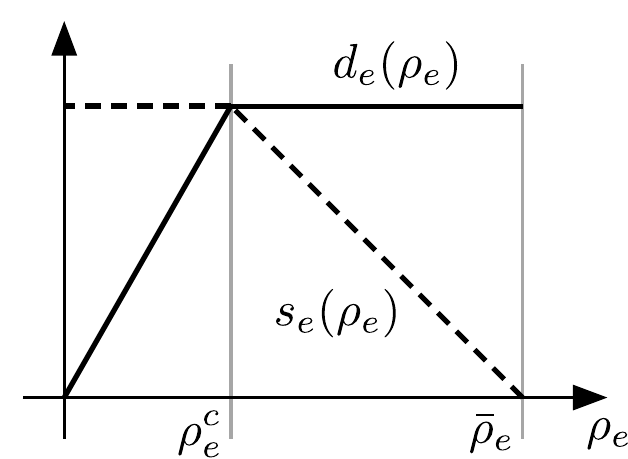}}
 		\end{picture}
		\caption{PWA demand and supply functions.}
		\label{fig:ctm_fd_a}
	\end{subfigure} ~ 
	\begin{subfigure}[b]{0.3\textwidth}
	  	\setlength{\unitlength}{0.1\textwidth}
  		\begin{picture}(10,9)
    			\put(0,0){\includegraphics[width=5cm]{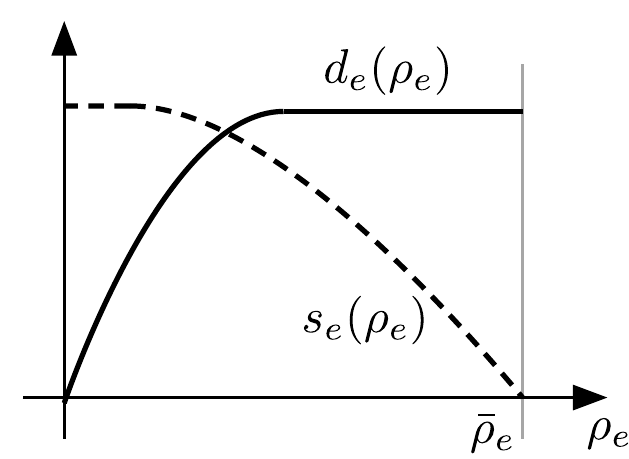}}
 		\end{picture}
		\caption{Concave demand and supply functions}	
		\label{fig:ctm_fd_b}
	\end{subfigure} ~ 
		\begin{subfigure}[b]{0.3\textwidth}
	  	\setlength{\unitlength}{0.1\textwidth}
  		\begin{picture}(10,9)
    			\put(0,0){\includegraphics[width=5cm]{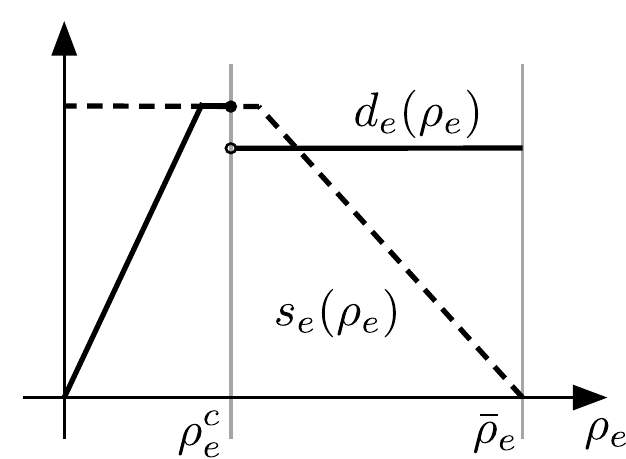}}
 		\end{picture}
		\caption{ Demand function with capacity drop.}	
		\label{fig:ctm_fd_c}
	\end{subfigure} 
	\caption{Different shapes of the demand and supply functions may be desirable in order to approximate real-world data. Figure (a) depicts the traditional PWA versions. Figure (b) depicts concave demand and supply functions, satisfying Assumption \ref{assumption:fd} for suitable $\Delta t$. Figure (c) shows a demand function with a capacity drop in congestion, which does \emph{not} satisfy Assumption \ref{assumption:fd}.}
	\label{fig:fd}
\end{figure}
 The classical, PWA demand and supply functions satisfy this assumption for suitable $\Delta t$, since $ \de = \min \left\{ v_e \rho_e,~ \frac{w_e}{v_e + w_e} \bar\rho_e \right\}$ and $ \se = \min \left\{ \frac{w_e}{ v_e + w_e} \bar\rho_e ,~ (\bar\rho_e - \rho_e) w_e \right\}$ with $v_e$ the free-flow speed and $w_e$ the congestion wave speed. The demand is non-decreasing and the supply is non-increasing by definition. Additionally, both functions are concave since they are defined as the pointwise minimum of affine functions. The condition on $\Delta t$ can be recognized as the stability condition $v_e \cdot \Delta t \leq l_e$ for all $e$, that arises if the CTM is derived as a discretization of the LWR model. Note that there is empirical evidence for the so-called capacity drop, that is, a non-monotone demand function as depicted in Figure \ref{fig:ctm_fd_c} \cite{kontorinaki2017first}. Assumption \ref{assumption:fd} excludes such behavior, but helps to keep the problem tractable and is in line with most of the literature on the FNC problem. Nevertheless, we will also demonstrate in Section \ref{sec:heuristic} on a freeway ramp metering example how our results can be used to design efficient heuristics that allow to target the effects of the capacity drop, even though the theoretical results do not directly extends to such models.

Using demand and supply functions, we are now ready to state the equations for the flows. For non-merging flows $\phi_e(t)$, that is, for $e \notin \N$, the flows are given as the minimum of upstream traffic demand and downstream supply of free space. Cases in which not all demand can be served are modeled using the first-in, first-out (FIFO) model: For every cell $e \notin \N$, the \emph{demand satisfaction} is computed as
\begin{equation*}
\kappa_e(t) = \min \left\{ 1, \underset{i \in \E^+(e)}{\min} \left\{ \frac{ s_i \big( \rho_i(t) \big)}{ \beta_{i,e} \cdot \de } \right\} \right\}.
\end{equation*}
Note that for any sink $s$, the set of downstream cells is empty, that is, $\E^+(s) = \varnothing$ and therefore $\kappa_s(t) = 1$. Using the definition of demand satisfaction, the flows are computed as
\begin{equation*}
\phi_e(t) = \kappa_e(t) \cdot \de \quad \forall e \in \E \setminus \N
\end{equation*}
The turning rates $\beta_{i,j}$ are respected at all times in the FIFO model. This implies that congestion in any one of the cells immediately downstream of a diverging junction will also reduce the flow to other such cells. We will have a closer look at this effect and its implications for optimal traffic control in Section \ref{sec:closer}. Recall also that we have assumed that turning rates are constant in time. This assumption is essential to in the FNC problem (defined subsequently in \eqref{eq:problem_original}, see also \cite{como2016convexity}), in which route choice is not externally controlled, but at discretion of the individual drivers. Time-varying turning rates could be exploited by the traffic network operator to indirectly alter the route choice and even the final destination of vehicles. For example, consider a freeway, where the turning rate of vehicles leaving at a particular offramp increases after a certain point in time. This potentially provides an incentive for the operator to create an artificial congestion upstream of the offramp until the increase in the turning rate occurs, as this would result in a larger number of vehicles exiting at the offramp once the corresponding turning rate increases. This effect is an artifact resulting from the use of time-varying turning rates instead of origin-destination maps for individual vehicles. It is intrinsic to the FNC problem and not specific to any particular optimal control approach. 

Similarly, care is required when modeling external demand. Recall that according to the conservation law \eqref{eq:ctm_conservation}, external inflows $w_e(t)$ into source cells $e \in \mathcal{S}$ are \emph{not} a priori constrained by the supply of free space. The alternative, to chose a model which only admits external demand that can be accommodated in a source cell of finite capacity and disregards surplus external demand, is problematic when the objective is the minimization of the TTS. It creates an incentive to cause congestion with the explicit goal of preventing parts of the external demand from entering the network. However, it turns out we can impose constraints on the maximal density in source cells nevertheless (Corollary \ref{lemma:onramps}, Section \ref{sec:application}). Such constraints prevent congestion from extending out of source cells into adjacent, unmodeled parts of the road network, as long as the optimal control problem remains feasible.

\begin{figure}[t] %%% FIGURE
	\centering
	\begin{subfigure}[b]{0.31\textwidth}
		\includegraphics[width=\textwidth]{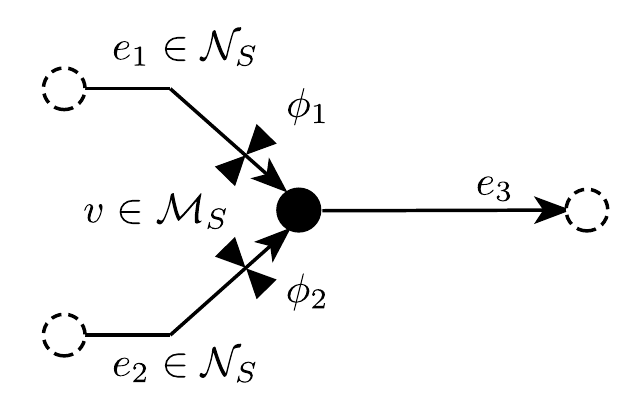}
		\caption{Symmetric junction.}
	\end{subfigure} ~	
	\begin{subfigure}[b]{0.31\textwidth}
		\includegraphics[width=\textwidth]{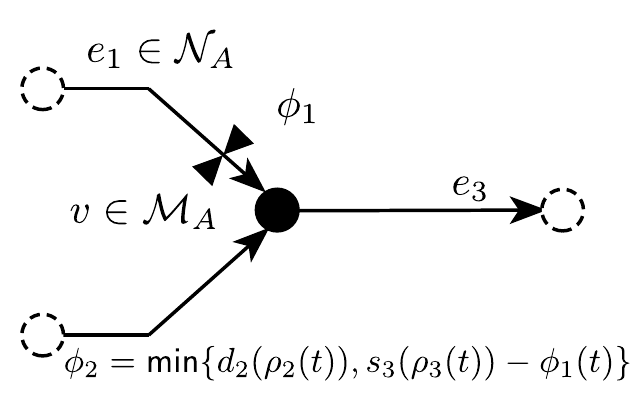}
		\caption{Asymmetric junction.}
	\end{subfigure} ~
	\begin{subfigure}[b]{0.31\textwidth}
		\includegraphics[width=\textwidth]{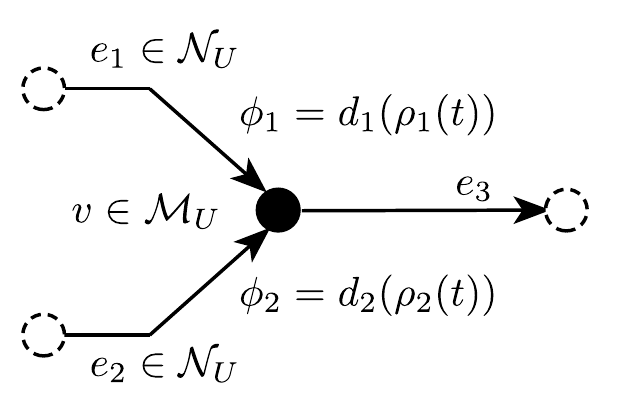}
		\caption{Uncongested junction.}
	\end{subfigure} ~
	\caption{Types of admissible merging junctions. Controlled flows are depicted with a control valve symbol.}
	\label{fig:merges}
\end{figure} %%% FIGURE

It remains to define the merging flows, that is, $\phi_e(t)$ for $e \in \N$. Existing models for uncontrolled junctions, such as Daganzo's priority rule \cite{daganzo1995cell} and the proportional-priority merging model \cite{kurzhanskiy2010active,coogan2016stability} differ in how available downstream supply is divided among upstream demand in case combined demand exceeds supply. In contrast to these approaches, we consider (partially) controlled merging junctions, which serve as actuators to optimize the operation of the traffic network. If all merging flows are controlled, we obtain a symmetric model:
\begin{itemize} %%% itemize
% ITEM (i)
	\item[(i)] In a \emph{symmetric junction} $v \in \V$, we assume that all merging flows are controlled, that is, the flows $\phi_e(t)$ for $\sigma_e = v$ are control inputs, subject to the constraints
\begin{align*}
0 \leq \phi_e(t) &\leq \de, \\[1ex]
\sum_{e \in \E} \beta_{i,e} \cdot \phi_e(t) &\leq s_i \big( \rho_i(t) \big).
\end{align*}
where $i$ denotes the unique edge downstream ($\tau_i = v$) of the merging junction $v$. We define the set of symmetric junctions $\M_S \subset \M$ and we denote the set of cells immediately upstream of a symmetric junction as $\N_S := \{e \in \E: \sigma_e \in \M_S \}$. 
\end{itemize}
Satisfaction of demand and supply constraints in a symmetric junction is ensured by the actuation constraints. To realize a symmetric junction in the real world, all incoming roads could be equipped with traffic lights to enable explicit control of each of the merging flows. Alternatively, the possibility of utilizing velocity control to regulate flows has also been discussed \cite{muralidharan2012optimal}. If actuation is realized via demand control in accordance with the actuation constraints, then the symmetric junction model is consistent with both Daganzo's priority rule and the proportional-priority merging model as described in Section \ref{sec:ex_network}. 

In practice, actuation of all merging flows is not always available. For example, in freeway ramp metering, it is assumed that onramp flows are controlled,  but there is typically no control of the mainline flow. Therefore, we introduce a second merging model suitable for certain metered onramps.
\begin{itemize}
% ITEM (ii)
	\item[(ii)] We assume that exactly two cells merge in an \emph{asymmetric junction} $v \in \M_A$, one cell $e$ modeling the mainline and one cell $j \in \N_A$ modeling a road merging into the mainline. The merging flow $\phi_j(t)$ is controlled, subject to the constraint $0 \leq \phi_j(t) \leq d_j \big( \rho_j(t) \big)$, but the flow on the mainline is determined as the minimum of demand and supply
\begin{equation}
\label{eq:ctm_flows2}
\phi_e(t) = \min \bigg\{ \de ~, ~ \frac{1}{\beta_{i,e}}  \Big( s_i \big( \rho_i(t) \big) - \beta_{i,j} \phi_j(t) \Big) \bigg\} ,
\end{equation}
where $i$ denotes the (unique) cell downstream of the asymmetric junction. 
\end{itemize}
The asymmetric junction model is based on the onramp model used in the Asymmetric CTM (ACTM) \cite{gomes2006optimal}. It prioritizes onramp inflows over mainline flows and requires the additional, implicit assumption that merging flows $\phi_j(t)$ for $j \in \N_A$ can always be accommodated on the mainline,
\begin{align}
\label{eq:asmerge_assumption}
s_i \big( \rho_i(t) \big) \geq \beta_{i,j} \cdot d_j \big( \rho_j(t) \big) \quad \forall j \in \N_A,~ \forall t,
\end{align}
to remain well-defined. We will make this assumption and its implications precise in Assumption \ref{assumption:merges} and Lemma \ref{lemma:invariance}. Intuitively, we trade off the need for active control of mainline flows with the additional assumption that mainline congestion does not extend onto onramps $j \in \N_A$. 

Finally, if congestion spillback of a junction can be excluded a priori, control of merging flows turns out not to be advantageous. 
\begin{itemize}
 % ITEM (iii)
\item[(iii)] Let $\M_U \subset \M$ denote the set of \emph{sub-critical} junctions\footnote{The index $U$ in $\M_U$ can be interpreted as ``uncongested by assumption".} and define the set of cells immediately upstream of such a junction as $\N_U := \{ e : \sigma_e \in \M_U \}$. We assume that all flows into sub-critical junctions are equal to the upstream demand
\begin{equation*}
\phi_e(t) = \de \quad \forall e \in \N_U
\end{equation*}
at all times.
\end{itemize} % end itemize
The sub-critical junction is a suitable model for intersections for which typical traffic volume is low, such that traffic densities remain below the critical densities in all relevant operating conditions. While such junctions may not typically be of interest in traffic control, they are included in the analysis to demonstrate that demand control of merging flows into such junctions is not advantageous\footnote{A consequence of Theorem \ref{theorem:CCTM}, see the subsequent discussion in Section \ref{sec:CCTM}.}. This complements the idea employed when defining the asymmetric junction, that a priori exclusion of congestion spillback can alleviate the need for active control of merging flows.
Ultimately, the introduction of asymmetric and sub-critical junctions allows for greater flexibility in the design of the system model used internally by an optimization-based control approach. It should be emphasized that one might encounter traffic networks with merging junctions that are neither actively controlled nor can be approximated using asymmetric or sub-critical junctions. The results in this work do not directly extend to such networks.
\begin{assumption} %%% ASSUMPTION
\label{assumption:merges}
The sets of symmetric junctions $\M_S$, asymmetric junctions $\M_A$ and sub-critical junctions $\M_U$ form a partition of the set of merging junctions $\M$. Source cells and cells immediately downstream of sub-critical junctions have infinite capacity, that is, $\bar\rho_e = + \infty$ and $s_e( \rho_e(t) ) = +\infty$. If asymmetric junctions are present in the network ($\M_A \neq \varnothing$), we only consider initial states $\rho(0)$ and external demand profiles $w(t)$ for $t \geq 0$, for which $s_i \big( \rho_i(t) \big) \geq \beta_{i,j} \cdot d_j \big( \rho_j(t) \big)$ for all $j \in \N_A$ and all $i \in \E$ holds for all $t \geq 0$ and any feasible control input sequence.
\end{assumption} %%% END assumption

The different categories of merging junction are depicted in Figure \ref{fig:merges}.Ensuring that certain states are not reachable in networks with asymmetric junctions is clearly hard to verify a priori. In that regard, Assumption \ref{assumption:merges} is similar to the corresponding assumption in the definition of the ACTM in \cite{gomes2006optimal}. Similarly to this reference, we will resort to an a posteriori verification of condition \eqref{eq:asmerge_assumption} for the optimal solution in our numerical study (Section \ref{sec:application}). 

For ease of notation, we also introduce the set of all cells $\mathcal{L} := \E \setminus \big( \N_S \cup \N_A \cup \N_U \big)$, for which flows are given as the minimum of supply and demand instead of being governed by particular merging rules. We collect the equations describing parts of the system model in the following succinct definition.
\begin{definition} %%% DEFINITION
\label{definition:ctm}
Consider a graph $\G = (V,\E)$ satisfying Assumption \ref{assumption:graph}, where each edge is equipped with a fundamental diagram satisfying Assumption \ref{assumption:fd} and merging junctions (and source cells) satisfy Assumption \ref{assumption:merges}. We define the \emph{CTM with controlled merging junctions} as the system with states $\rho_e(t)$ for $e \in \E$ and inputs $\phi_e(t)$ for $e \in \N_S \cup \N_A$. The state evolves as
\begin{align}
\rho_e(t+1) = \rho_e(t) + \frac{\Delta t}{l_e} \cdot \left( \sum_{i \in \E} \beta_{e,i} \phi_i(t) ~ - \phi_e(t) + w_e(t) \right) \quad & \forall e \in \E
\label{eq:ctm_density}
\end{align}
with
\begin{subequations}
\begin{align}
\phi_e(t) &= \min \left\{ \de,  ~\underset{i \in \E^+(e)}{\min} \left\{ \frac{ s_i \big( \rho_i(t) \big) - \sum_{j \in \N_A} \beta_{i,j} \cdot \phi_j(t) }{ \beta_{i,e}  } \right\}  \right\}  \quad &&\forall e \in \mathcal{L} ~, \label{eq:ctm_flow_a} \\
\phi_e(t) &= \de &&\forall e \in \N_U ~. \label{eq:ctm_flow_b}
\end{align}
\label{eq:ctm_flow}
\end{subequations}
The inputs are subject to the constraints
\begin{subequations}
\begin{align}
0 \leq \phi_e(t) &\leq \de && \forall e \in \N_S \cup \N_A ~, \label{eq:ctm_constraints_demand} \\ 
\sum_{i \in \E} \beta_{e,i} \cdot \phi_i(t) &\leq \se && \forall e : \tau_e \in \M_S ~. \label{eq:ctm_constraints_supply}
\end{align}
\label{eq:ctm_constraints}
\end{subequations}
\end{definition} %%% DEFINITION

The following technical Lemma ensures that the system evolution is well-defined. 
\begin{lemma} %%% LEMMA
\label{lemma:invariance}
Consider a network with controlled merging junctions, with initial state $\rho(0)$ and external demand pattern $w(t)$ ($t \geq 0$) satisfying Assumptions \ref{assumption:graph}, \ref{assumption:fd} and \ref{assumption:merges}. Define $\mathbb{P} := \prod_{e \in \E} [0, \bar\rho_{e}]$ , with $\bar\rho_e = + \infty$ for cells of infinite capacity. Then, for every reachable $\rho(t) \in \mathbb{P}$, there exists a feasible input $\phi_e(t)$. In addition, for every feasible input $\phi_e(t)$, the subsequent state satisfies $\rho(t+1) \in \mathbb{P}$.
\end{lemma}
The proof is presented in \ref{appendix:lemma_invariance}. To complete the problem description, it remains to introduce the control objective. A natural objective in traffic control is to minimize the total time spent (TTS) on the road
\begin{align*}
\TTS = \Delta t \cdot \sum_{t=1}^{T} \sum_{e \in \E} l_e \cdot \rho_e(t)
\end{align*}
over the horizon $t \in \{1,2,\dots,T\}$, that can represent a rush hour period or an entire day \cite{papageorgiou2003review,gomes2006optimal,como2016convexity}. The FNC problem for the CTM with controlled merging junctions is then defined as
\begin{equation} %%% EQUATION
\label{eq:problem_original}
\begin{array}{rrl}
  \mathcal{P}_\text{CTM}^* = & \underset{\phi(t), \rho(t)}{\text{minimize}} & \TTS  \\ [2ex]
 & \text{subject to} & \text{CTM dynamics \eqref{eq:ctm_density} and \eqref{eq:ctm_flow}, and constraints \eqref{eq:ctm_constraints}} \\[1ex]
  & & \rho(0) ~\text{given.}
\end{array}
\end{equation}
The FNC problem assumes that predictions for the external traffic demands $w_e(t)$ are available for the optimization horizon $t \in \{0, \dots, T-1\}$. In practice, such predictions are highly uncertain and thus the control inputs computed by solving \eqref{eq:problem_original} should not be applied in open loop. This work focusses exclusively on the efficient solution of \eqref{eq:problem_original}, but it is emphasized that the control inputs obtained by solving the FNC problem should be applied in a receding horizon fashion to mitigate the effects of demand and model uncertainty. The FNC problem is non-convex, due to the nonlinear equality constraints describing the flows, but we aim to find a tight convex relaxation. To facilitate analysis, we will derive an equivalent system representation in the following section.

\section{Concave, state-monotone reformulation} \label{sec:CCTM} %%% SECTION =========================================== %

In this section, we introduce an equivalent system representation of the traffic model based on the \emph{cumulative flow}
\begin{equation}
\label{eq:cumulative_flow}
\Phi_e(t) := \Delta t \cdot \sum_{\tau = 0}^{t-1} \phi_e( \tau ) ~,
\end{equation}
that leaves a particular cell $e$ over the horizon $[0, t-1]$. The cumulative flows will serve as states, but certain flows are actively controlled. Therefore, we also introduce \emph{cumulative inputs} $\varphi_e(t) := \Delta t \cdot \sum_{\tau = 0}^{t} \phi_e( \tau )$ for controlled flows, that is, for $e \in \N_S \cup \N_A$. Note that the summation to compute $\varphi_e(t)$ is up to $\tau = t$ and therefore, the controlled flow $\phi_e(t) = \frac{1}{\Delta t} \cdot \big( \varphi_e(t) - \Phi_e(t) \big)$ can be computed given the cumulative flow $\Phi_e(t)$ and the cumulative input $\varphi_e(t)$. For ease of notation, we also introduce the cumulative, external inflow $ W_e(t) := \Delta t \cdot \sum_{\tau = 0}^{t-1} w_e( \tau ) $. Given a trajectory of cumulative flows, the corresponding densities and flows can be computed as
\begin{subequations}
\begin{align}
\phi_e(t) &= \frac{1}{\Delta t} \big( \Phi_e(t+1) - \Phi_e(t) \big) & \quad \forall e \in \E  ~, \\
\rho_e(t) &= \rho_e(0) + \frac{1}{l_e} \cdot \left( \sum_{i \in \E} \beta_{e,i} \Phi_{i}(t) ~ - \Phi_e(t) + W_e(t) \right) & \quad \forall e \in \E ~.\label{eq:density_transformation}
\end{align}
\label{eq:state_transformation}
\end{subequations}
It is helpful to introduce the negative cumulative flows $\hat \Phi_e(t) := - \Phi_e(t)$ for the controlled flows $e \in \N_S \cup \N_A$ as separate states. Our aim is to apply Theorem \ref{theorem:csm} to the FNC problem and the negative cumulative flows will be instrumental in defining a state-monotone reformulation of the CTM. The symbol $\Phi(t)$ is used to denote the vector composed of $\Phi_e(t)$ for all $e \in \E$ and $\hat \Phi_e(t)$ for $e \in \N_S \cup \N_A$. Before proceeding, we also introduce two auxiliary functions. First, the \emph{cumulative demand}
\begin{align*}
D_e \big( \Phi(t) \big) := \Phi_e(t) + \Delta t \cdot d_e \Big( \rho_e \big( \Phi(t) \big) \Big)
\end{align*}
is defined for all $e \in \E$, where $\rho_e \big( \Phi(t) \big)$ is computed according to the conservation law \eqref{eq:density_transformation}. Second, the \emph{cumulative supply}
\begin{align*}
S_{i,e} \big( \Phi(t) , \varphi(t) \big) &:= \Phi_e(t) + \frac{\Delta t}{\beta_{i,e}} \cdot \bigg( s_{i,e} \big( \Phi(t) \big) - \sum_{j \in \N_A} \beta_{i,j} \cdot \underbrace{ \frac{\varphi_j(t) - \Phi_j(t)}{\Delta t} }_{= \phi_j(t)} \bigg)
\end{align*}
is defined for any cell $e \notin \big( \N_S \cup \N_A \cup \N_U \big) $ and downstream cell $i$ such that $\beta_{i,e} \neq 0$, with
\begin{align*}
s_{i,e} \big( \Phi(t) \big) := s_i \Bigg(  \underbrace{ \rho_i(0) + \frac{1}{l_i} \cdot \bigg( \beta_{i,e} \Phi_e(t) + \sum_{j \in \N_A} \beta_{i,j} \Phi_{j}(t) ~ - \Phi_i(t) + W_i(t) \bigg) }_{= \rho_i(t)} \Bigg) .
\end{align*} 
In the latter equation, $s_i( \cdot )$ is simply the supply function of cell $i$ and its argument is equal to the density $\rho_i(t)$. 

\begin{lemma}
\label{lemma:CCTM}
Consider the CTM with controlled merging junctions and initial state $\rho(0)$.\footnote{The initial density is necessary to compute $\rho_e\big( \Phi(t) \big) = \rho_e(0) + \frac{1}{l_e} \cdot \left( \sum_{i \in \E} \beta_{e,i} \Phi_{i}(t) ~ - \Phi_e(t) + W_e(t) \right)$, which in turn is needed to evaluate $D_e\big( \Phi(t) \big)$. $\rho_e(0)$ is also necessary to evaluate $s_{i,e}(\Phi(t))$, which is used to define $S_{i,e} \big( \Phi(t), \varphi(t) \big)$.} The system evolution can equivalently be described using the system with state $\Phi(t)$ and input $\varphi_e(t)$ for $e \in \N_S \cup \N_A$. The state evolves as 
\begin{subequations}
\label{eq:csm_flow}
\begin{align}
& \Phi_e(t+1) = \min \left\{ D_e \big( \Phi(t) \big), \min_{i \in \E^+(e)} S_{i,e} \big( \Phi(t), \varphi(t) \big) \right\} &&\forall e \in \mathcal{L} , \label{eq:csm_flow_a} \\
& \Phi_e(t+1) = D_e \big( \Phi(t) \big) &&\forall e \in \N_U , \label{eq:csm_flow_b} \\[1ex]
& \Phi_e(t+1) = \varphi_e(t) &&\forall e \in \N_S \cup \N_A , \label{eq:csm_flow_c} \\[1ex]
& \hat \Phi_e(t+1) = -\varphi_e(t) &&\forall e \in \N_S \cup \N_A , \label{eq:csm_flow_d}
\end{align}
\end{subequations}
starting from the initial state $\Phi_e(0) = 0$ and $\hat \Phi_e(0) = 0$. The controlled flows are subject to demand and supply constraints 
\begin{subequations}
\label{eq:csm_constraints}
\begin{align}
-\hat \Phi_e(t) \leq \varphi_e(t) &\leq D_e \big( \Phi(t) \big) && \forall e \in \N_S \cup \N_A  , \label{eq:csm_constraints_demand} \\
\sum_{i \in \E} \beta_{e,i} \cdot \varphi_i(t) &\leq \Delta t \cdot s_e \Big( \rho_e \big( \Phi(t) \big) \Big) + \sum_{i \in \E} \beta_{e,i} \cdot \Phi_i(t) && \forall e : \tau_e \in \M_S . \label{eq:csm_constraints_supply} 
\end{align}
\end{subequations}
\end{lemma}

\begin{proof}
Multiplying the flow constraints \eqref{eq:ctm_flow_a}, \eqref{eq:ctm_flow_b} and the demand constraint \eqref{eq:ctm_constraints_demand} with $\Delta t$, then adding $\Phi_e(t)$ on both sides and substituting according to the transformation equations \eqref{eq:cumulative_flow} and \eqref{eq:state_transformation} yields the dynamics of the uncontrolled cumulative flows \eqref{eq:csm_flow_a}, \eqref{eq:csm_flow_b} and the cumulative demand constraint \eqref{eq:csm_constraints_demand}, respectively. Similarly, the supply constraint \eqref{eq:ctm_constraints_supply} can be transformed in the cumulative supply constraint \eqref{eq:csm_constraints_supply} by multiplying the former equation with $\Delta t$, then adding the term $\sum_{i \in \E} \beta_{e,i} \Phi_i(t)$ on both sides and performing the appropriate substitutions. Given a trajectory in cumulative flows, the conservation law \eqref{eq:ctm_density} follows directly from the transformation equation \eqref{eq:density_transformation} which defines the densities as functions of the cumulative flows. In the reverse direction, the system dynamics of the cumulative, controlled flows \eqref{eq:csm_flow_c} and \eqref{eq:csm_flow_d} follow directly from the definitions of the cumulative inputs $\varphi_e(t)$ and the negative cumulative flows $\hat \Phi_e(t)$.
\end{proof}

We refer to the system described by \eqref{eq:csm_flow} and \eqref{eq:csm_constraints} as the \emph{cumulative cell transmission model} (CCTM). Note that the negative controlled flows $\hat \Phi_e(t)$ have been used to define the lower bound in \eqref{eq:csm_constraints_demand}. In Lemma \ref{lemma:CCTM}, ``equivalent" means that any feasible trajectory $\big( \rho(t), \phi(t) \big)$ of the CTM with controlled merging junctions can be transformed into a feasible trajectory $\big( \Phi(t), \varphi(t) \big)$ of the CCTM (and vice versa).
The cumulative flows used in the definition of the CCTM are similar in spirit to the concept of cumulative arrivals and departures used in network calculus. Network calculus originated in the analysis of communication networks, but has also found application in the control of traffic networks. For example, \cite{varaiya2013max} considers a store-and-forward model which does not include congestion spillback effects and studies stabilization of network queues in the long term. More recently, a reformulation similar to the CCTM for the special case of a freeway segment with only onramp and off-ramp junctions was employed in \cite{schmitt2017sufficient} to derive optimality conditions for distributed, non-predictive ramp metering. 

It turns out that the CCTM is a concave, state-monotone system -- but only if the negative cumulative flows $\hat \Phi_e(t)$ are used instead of $-\Phi_e(t)$ to define the lower bound on $\varphi_e(t)$ in \eqref{eq:csm_constraints_demand}. In the following, we will analyze concavity and monotonicity of the equations defining the CCTM.
\begin{lemma} %%% LEMMA
\label{lemma:demand}
For any $e \in \E$, the \emph{cumulative demand} $D_e \big(\Phi(t) \big)$
is concave and monotone in $\Phi(t)$.
\end{lemma}
The proof is provided in \ref{appendix:lemma_demand}. A similar result holds for the cumulative supply.
\begin{lemma} %%% LEMMA
\label{lemma:supply}
For any cell $e \notin \mathcal{L} $ and downstream cell $i$ such that $\beta_{i,e} \neq 0$, the \emph{cumulative supply} $S_{i,e} \big( \Phi(t), \varphi(t) \big)$
is jointly concave in $\Phi(t)$ and $\varphi(t)$ and state-monotone, that is, it is monotone in $\Phi(t)$.
\end{lemma}
The proof is provided in \ref{appendix:lemma_supply}. We now use the preceding lemmas to analyze the system dynamics of the CCTM.
\begin{lemma} %%% LEMMA
\label{lemma:csm_dynamics}
The CCTM dynamics \eqref{eq:csm_flow} are concave and state-monotone.
\end{lemma}
\begin{proof} %%% PROOF
The state evolution for cells $e \in \mathcal{L}$ is given as the point-wise minimum
\begin{align*}
\Phi_e(t+1) &= \min \left\{ D_e \big( \Phi(t) \big), \min_{i \in \E^+(e)} S_{i,e} \big( \Phi(t), \varphi(t) \big) \right\}
\end{align*}
over the cumulative demand $D_e \big( \Phi(t) \big)$ and the cumulative supply $S_{i,e} \big( \Phi(t), \varphi(t) \big)$. Concavity and state monotonicity of cumulative demand and supply have been shown in Lemmas \ref{lemma:demand} and \ref{lemma:supply}, respectively. Taking the point-wise minimum over finitely many functions preserves both concavity \cite{boyd2004convex} and monotonicity of the arguments. Hence, the system dynamics are concave and state-monotone. For cells $e \in \N_U$ upstream of sub-critical junctions, the cumulative flow $\Phi_e(t+1) = D_e \big( \Phi(t) \big)$ is equal to the cumulative demand and is thus concave and state-monotone. The system equations for $\Phi_e(t+1)$ and $\hat \Phi_e(t+1)$ for $e \in \N_S \cup \N_A$ are linear in the input and do not depend on the state. Therefore, they are also concave and state-monotone.
\end{proof} % end proof
The limitations of introducing additional, negative states $\hat \Phi_e(t)$ to retain state monotonicity of the lower bound on the controlled flows \eqref{eq:csm_constraints_demand} become apparent in the previous analysis of the system dynamics: both the right-hand sides (RHS) of \eqref{eq:csm_flow_c} and \eqref{eq:csm_flow_d} need to be concave and state-monotone, even though the RHS of \eqref{eq:csm_flow_d} is the negative of the RHS of \eqref{eq:csm_flow_c}. In the particular case of the CCTM, this is possible since the RHS is linear in the inputs and does not depend on the system states.
\begin{lemma} %%% LEMMA
\label{lemma:csm_constraints}
The CCTM constraints \eqref{eq:csm_constraints} are concave and state-monotone.
\end{lemma}
\begin{proof} %%% begin PROOF
To analyze monotonicity of the constraints, we will transform all constraints into the form $g_t(x(t),u(t)) \geq 0$ to match the conventions used in Theorem \ref{theorem:csm}. The left-hand side (LHS) of the constraint preventing negative controlled flows, $\varphi_e(t) + \hat \Phi_e(t) \geq 0$ for all $e \in \N_S \cup \N_A$,
is affine and hence concave. In addition, it is monotone in $\hat \Phi_e(t)$ and hence state-monotone. The LHS of the demand constraints, $D_e \big( \Phi(t) \big) - \varphi_e(t) \geq 0$ for all $e \in \N_S \cup \N_A$,
is concave and state-monotone, since it is the sum of the cumulative demand $D_e \big( \Phi(t) \big)$, which is concave and state-monotone (Lemma \ref{lemma:demand}), and the negative of the controlled cumulative flow, which is also concave and state-monotone (as it is independent of the state) as well. For the supply constraints
\begin{align*}
\Delta t \cdot s_e \Big( \rho_e \big( \Phi(t) \big) \Big) + \sum_{i \in \E} \beta_{e,i} \cdot \big( \Phi_i(t) - \varphi_e(t) \big) \geq 0 &\quad \forall e : \tau_e \in \M_S,
\end{align*}
we first focus on concavity. The density $\rho_e\big( \Phi(t) \big)$ is an affine function of $\Phi(t)$ according to the conservation law \eqref{eq:state_transformation}. The supply function $s_e( \cdot ) $ is concave in the density by Assumption \ref{assumption:fd} and hence, is concave in $\Phi(t)$. Concavity of the LHS follows since it is sum of concave and affine functions. State monotonicity of the LHS can be verified using similar reasoning as in the proof of Lemma \ref{lemma:supply}. This is detailed in \ref{appendix:proof_constraint_supply}.
\end{proof} %%% end PROOF

It remains to express the objective function in terms of cumulative states. In the CCTM, we consider the maximization of linear, stage-wise objective functions\footnote{It is straightforward to extend the results to concave, state-monotone objective functions that are not necessarily defined stage-wise. However, relevant objective functions for the FNC problem, such as TTS, can be expressed as stage-wise, linear functions. We restrict our attention to this case for simplicity of exposition.} of the form
\begin{align}
\label{eq:csm_objective}
c \big( \Phi(t) \big) = \sum_{e \in \E} c_e \cdot \Phi_e(t),
\end{align}
with $c_e \geq 0$. This objective function is linear in states and inputs, and hence concave. Also, $\frac{\partial}{\partial \Phi_e} c ( \Phi(t)) = c_e \geq 0$ implies that it is state-monotone. The choice $\hat c_e := 1 - \sum_{i \in \E} \beta_{i,e} \geq 0$ makes the stage-wise objective \eqref{eq:csm_objective} equal to the total discharge flow that leaves the network in time step $t$. One can use the objective function \eqref{eq:csm_objective} to encode minimization of TTS. We can verify that
\begin{align*}
\TTS &= \sum_{t=1}^{T} \sum_{e \in \E} l_e \cdot \rho_e(t) \\
 % &= \sum_{t=1}^{T} \sum_{e \in \E} l_e \cdot \left( \rho_e(0) + \frac{1}{l_e} \cdot \left( \sum_{j \in \E ^-(e)} \beta_{je} \Phi_{j}(t) ~ - \Phi_e(t) + W_e(t) \right) \right) \\
 &= \sum_{t=1}^{T} \sum_{e \in \E} \left( l_e \cdot \rho_v(0) - \left(\Phi_e(t) - \sum_{j \in \E} \beta_{e,j} \Phi_{j}(t) \right) ~ + W_e(t)  \right) \\
 &= \sum_{t=1}^{T} \sum_{e \in \E} \left( l_e \cdot \rho_v(0) - \left(\Phi_e(t) - \sum_{i \in \E} \beta_{i,e} \Phi_{e}(t) \right) ~ + W_e(t)  \right) \\
 &= \underbrace{\sum_{t=1}^{T} \sum_{e \in \E} \left(  l_e \cdot \rho_v(0) + W_e(t) \right)}_{ := C_{\text{W}}} - \sum_{t=1}^{T} \sum_{e \in \E} \hat c_e \Phi_e(t)
\end{align*}
where the equality in the third line follows from rearranging flows within the network in the summation over $e \in \E$. Note that $C_{\text{W}}$ is a constant that does not affect the set of minimizers. Hence, maximization of the stage-wise discharge is equivalent to minimizing TTS over the whole horizon, in the sense that the set of optimizers is identical. The relationship between total discharge flows and TTS has already been pointed out in \cite{papageorgiou2003review,gomes2006optimal}, while concavity and monotonicity of the TTS objective in the density state is discussed in \cite{como2016convexity}.

We are now ready to state our main result:
\begin{theorem} %%% THEOREM
\label{theorem:CCTM}
Consider a traffic network modeled by the CTM defined on a graph satisfying Assumption \ref{assumption:graph}, with fundamental diagrams satisfying Assumption \ref{assumption:fd} and with merging junctions satisfying Assumption \ref{assumption:merges}. The FNC problem for the CTM with controlled merging junctions \eqref{eq:problem_original} is equivalent to the convex, \emph{relaxed FNC problem}
\begin{equation} %%% EQUATION
\label{eq:problem_relaxed}
\begin{array}{rrll}
  \mathcal{R}_\text{CTM}^* = & \underset{\phi(t), \rho(t)}{\text{minimize}} & \Delta t \cdot \sum_{t = 1}^{T} \sum_{e \in \E} l_e \cdot \rho_e(t) \\ [2ex]
 & \text{subject to} & \rho_e(t+1) = \rho_e(t) + \frac{\Delta t}{l_e} \cdot \left( \sum_{i \in \E} \beta_{e,i} \phi_i(t) ~ - \phi_e(t) + w_e(t) \right) \quad & \forall e \in \E \\ [1ex]
 & & \phi_e(t) \leq \de \quad & \forall e \in \E \\ [1ex]
 & & \sum_{i \in \E} \beta_{e,i} \cdot \phi_i(t) \leq \se & \forall e \in \E \\[2ex]
 & & \phi_e(t) \geq 0 & \forall e \in \N_S \cup \N_A \\ [1ex]
  & & \rho(0) ~\text{given,}
\end{array} 
\end{equation}
in the sense that the objective values are equal and any optimizer of \eqref{eq:problem_relaxed} can be used to compute a solution of the original problem \eqref{eq:problem_original}.
\end{theorem} % theorem
\begin{proof} %%% PROOF
Consider an optimizer $\big( \tilde \rho(t), ~\tilde \phi(t) \big)$ of the relaxed FNC problem \eqref{eq:problem_relaxed}. The corresponding cumulative flows $\tilde \Phi(t)$ and inputs $\tilde \varphi(t)$, computed according to the definition \eqref{eq:cumulative_flow}, are an optimizer of the equivalent optimization problem 
\begin{equation} 
\label{eq:problem_cctm_relaxed}
\begin{array}{rrlr}
 \mathcal{R}_\text{CCTM}^* = & C_{\text{W}} - \underset{\Phi(t), \rho(t)}{\text{maximize}} & \sum_{t=1}^{T} \sum_{e \in \E} \hat c_e \Phi_e(t) \\ [2ex]
 & \text{subject to} & \Phi_e(t+1) \leq \min \left\{ D_e \big( \Phi(t) \big), \underset{i \in \E^+(e)}{\min} S_{i,e} \big( \Phi(t), \varphi(t) \big) \right\} \hspace{-0.5cm} & \forall e \in \mathcal{L} \\ [2ex]
 & &\Phi_e(t+1) \leq D_e \big( \Phi(t) \big) &\forall e \in \N_U \\ [1ex]
 & &\Phi_e(t+1) \leq \varphi_e(t) &\forall e \in \N_S \cup \N_A \\ [1ex]
 & &\hat \Phi_e(t+1) \leq -\varphi_e(t) &\forall e \in \N_S \cup \N_A \\ [1ex]
 & & \text{CCTM constraints \eqref{eq:csm_constraints}} \\ [1ex]
 & & \rho(0) ~\text{given, } \Phi_e(0) = 0, \hat \Phi_e(0) = 0 .
\end{array}
\end{equation}
Details proving the equivalence of the relaxed problem \eqref{eq:problem_relaxed} and its counterpart in terms of the cumulative flows \eqref{eq:problem_cctm_relaxed} are given in \ref{appendix:backtransformation}. 

The system dynamics and constraints of the CCTM are concave and state-monotone (Lemma \ref{lemma:csm_dynamics} and \ref{lemma:csm_constraints} respectively), and we have shown that the objective function encoding the total discharge flow is also concave and state-monotone. Hence, we can apply Theorem \ref{theorem:csm}, which implies that the solution of the relaxed problem can be used to compute an optimizer of 
\begin{equation}
\label{eq:problem_cctm_original}
\begin{array}{rrll}
 \mathcal{P}_\text{CCTM}^* =  & C_\TTS - \underset{\Phi(t), \rho(t)}{\text{maximize}} & \sum_{t=1}^{T} \sum_{e \in \E} \hat c_e \Phi_e(t) \\ [2ex]
 & \text{subject to} & \text{CCTM dynamics \eqref{eq:csm_flow} and constraints \eqref{eq:csm_constraints}} \\ [1ex]
 & & \rho(0) ~\text{given, } \Phi_e(0), \hat \Phi_e(0) = 0 = 0 .
\end{array}
\end{equation}
by forward simulation. We denote the resulting, optimal trajectory by $\Phi^*(t)$. By equivalence of the FNC problem \eqref{eq:problem_original} and its counterpart in cumulative flows \eqref{eq:problem_cctm_original} according to Lemma \ref{lemma:CCTM}, we can compute a feasible trajectory $\big( \rho^*(t), ~\phi^*(t) \big)$ of the FNC problem by applying the inverse transformation \eqref{eq:state_transformation}. It follows that $\mathcal{R}_\text{CTM}^* = \mathcal{R}_\text{CCTM}^* = \mathcal{P}_\text{CCTM}^* \geq \mathcal{P}_\text{CTM}^*$. In addition, we have that $\mathcal{R}_\text{CTM}^* \leq \mathcal{P}_\text{CTM}^*$ since \eqref{eq:problem_relaxed} is a relaxation of the FNC problem \eqref{eq:problem_original}. Therefore $\mathcal{R}_\text{CTM}^* = \mathcal{P}_\text{CTM}^*$ which proves the first claim of the theorem. In addition, $\phi_e^*(t) = \tilde \phi_e(t)$ for all controlled flows $e \in \N_S \cup \N_A$ and hence the second claim of the theorem follows: a solution to the FNC problem \eqref{eq:problem_original} can be obtained by forward simulation of the optimal inputs of the relaxed CTM problem \eqref{eq:problem_relaxed}. 
\end{proof}
It should be emphasized that while the CCTM is crucial for the proof, it is \emph{not} necessary to perform the transformation from flows to cumulative flows when applying Theorem \ref{theorem:CCTM}. Instead, one can solve relaxation \eqref{eq:problem_relaxed} and perform the forward simulation using the CTM with controlled merging junctions. Note also that the relaxed FNC problem \eqref{eq:problem_relaxed} corresponds to the ``natural" relaxation of the FNC problem, where the flows are not restricted lie on the fundamental diagram, but are only constrained to its convex hypograph (the set of points below the graph). All solutions of the FNC problem, regardless of the actuation scheme, are feasible in the natural relaxation \eqref{eq:problem_relaxed}, as discussed in \cite{como2016convexity}. Therefore, Theorem \ref{theorem:CCTM} implies that if control of all merging flows (as for symmetric junctions) is available and Assumptions \ref{assumption:graph} and \ref{assumption:fd} hold, further demand control in other cells is not advantageous in reducing TTS. Furthermore, additional assumptions (as made for asymmetric and sub-critical junctions) can alleviate the need for control of certain merging flows. Intuitively, the purpose of control for merging junctions is to prioritize certain flows over others, trading off the consequences of congestion propagating into different upstream cells. The relaxed FNC problem implicitly assumes that prioritizing individual merging flows is possible (or that the merging junction is sub-critical, that is, uncongested, and hence no trade off is necessary), however, it is agnostic about the means by which this is accomplished. The definitions of symmetric junctions and asymmetric junctions describe two different actuation schemes which make prioritization of individual merging flows possible.

\section{State monotonicity of merging and diverging junctions} \label{sec:closer} %%% SECTION ===================================== %

Monotonicity of the CTM has already been studied comprehensively. In particular, the CTM of a freeway segment with only onramp and off-ramp junctions is monotone if a simplified onramp model (similar to the asymmetric junction model) is used \cite{gomes2008behavior}. However, if the CTM is expressed with the densities as states, the dynamics of FIFO diverging junctions are \emph{not} monotone, though they satisfy a mixed-monotonicity property \cite{coogan2015compartmental,coogan2016stability}. There also exist monotone models for diverging junctions, which have been used to employ results from monotone system theory to analyze stability of traffic networks \cite{lovisari2014stability,lovisari2014stability2}. However, these models do \emph{not} preserve the turning rates.

In this work, we showed that the FIFO diverging model can be described by a state-monotone system that uses the cumulative flows as states. This means in particular that congestion spillback, the effect which \cite{coogan2015compartmental} highlight as important to consider for achieving benefits via active freeway control, are present in our model. We illustrate the effect in Example \ref{example:diverge}, using both density and cumulative flow states. Note that in the following example and in subsequent sections, we will make use of a shorthand notation for indices: whenever cells are labeled using integers, e.g., $\E = \{e_1, e_2, \dots, e_n\}$, we will use the shorthand notation $\rho_1(t),~ \rho_2(t) \dots \rho_n(t)$ with only the integer as index instead of $\rho_{e_1}(t),~ \rho_{e_2}(t) \dots \rho_{e_n}(t)$ to denote any quantities that are indexed by cells.

\begin{figure}[t] %%% FIGURE
	\centering
	\begin{subfigure}[b]{0.28\textwidth}
		\includegraphics[width=\textwidth]{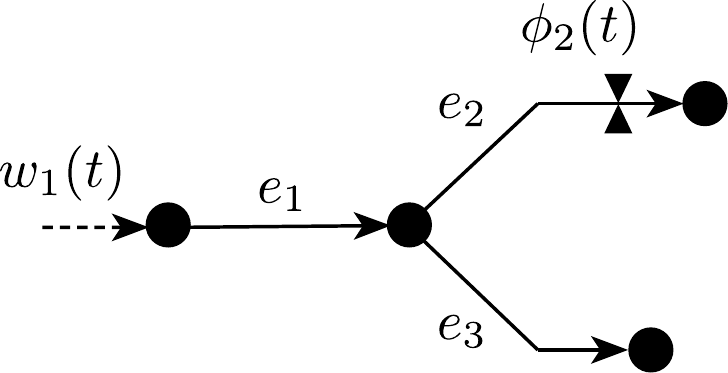}
		\vspace{0.05cm}
		\caption{Simple network with a FIFO diverging junction.}
		\label{fig:diverge_structure}
	\end{subfigure} ~~
	\begin{subfigure}[b]{0.33\textwidth}
		\includegraphics[width=\textwidth]{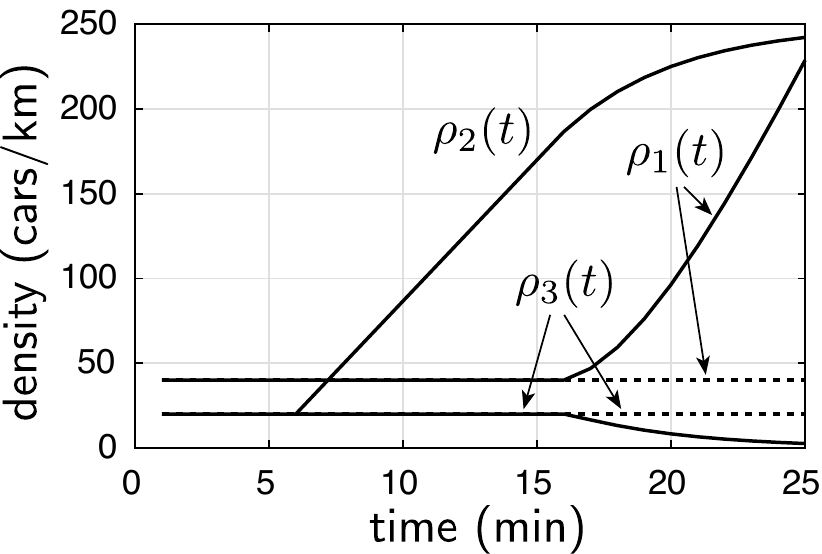}
		\caption{Evolution of traffic densities.}
		\label{fig:diverge_density}
	\end{subfigure} ~~
	\begin{subfigure}[b]{0.33\textwidth}
		\includegraphics[width=\textwidth]{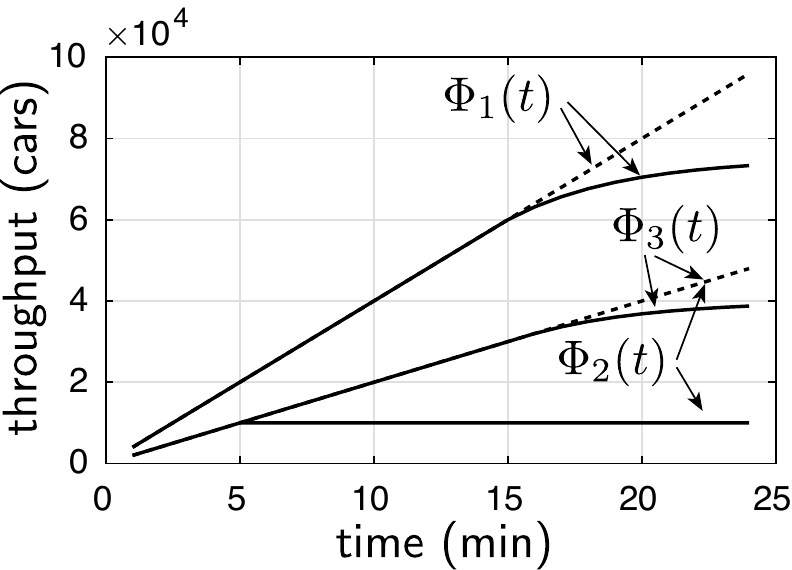}
		\caption{Evolution of cumulative flows.}
		\label{fig:diverge_throughput}
	\end{subfigure}
	\caption{A reduction in the controlled flow $\phi_2(t)$ leads to a non-monotone response in terms of the traffic densities, but a monotone response in terms of the cumulative flows. The trajectories obtained if $\phi_2(t)$ is not reduced are shown in dashed.}
	\label{fig:example_diverge}
\end{figure} %%% FIGURE
\begin{example}
\label{example:diverge}
Consider a simple network as depicted in Figure \ref{fig:example_diverge}. For simplicity, we chose a triangular fundamental diagram with identical parameters among all cells, in particular $l_e = 1$km, $v_e = 100$km/h, $w_e = 25$km/h and $\bar\rho_e = 250\text{cars/km}$. The turning rates of the diverging junction are $\beta_{2,1} = \beta_{3,1} = \frac{1}{2}$ and the external inflow $w_1(t) = 4000$cars/h is constant. The system starts in a free-flow equilibrium.\footnote{Here, ``equilibrium" means that the densities are constant, but not the cumulative flows.} Equilibrium states are depicted by dashed lines in Figure \ref{fig:diverge_density} and \ref{fig:diverge_throughput}. We compare this baseline scenario against a scenario where cell $e_2$ upstream of the diverging junction experiences congestion: at $t = 6$min, the controlled flow $\phi_2(t)$ is reduced to zero (or equivalently, $\Phi_2(t) = \Phi_2(6 \text{min})$ for all $t \geq 6$min), which leads to an increase of $\rho_2(t)$ first and $\rho_1(t)$ subsequently, depicted by solid lines in Figure \ref{fig:diverge_density}. A state-monotone system model should preserve the ordering of trajectories, but the density $\rho_3(t)$ falls \emph{below} the corresponding baseline trajectory, thus demonstrating that the model is not state monotone in the densities. By contrast, the reduction in $\Phi_2(t)$ reduces the growth of both $\Phi_1(t)$ and $\Phi_3(t)$ in comparison to the baseline trajectory, as it can be seen in the solid lines in Figure \ref{fig:diverge_throughput}. This bevavior is consistent with a state-monotone model.
\end{example}

While the dynamics of FIFO diverging junctions are state-monotone in the cumulative flows, the dynamics of congested, uncontrolled merging junctions are not. This is the reason for restricting admissible types of merging junctions to the three types introduced in Section \ref{sec:networks}. We illustrate non-monotonicity of a congested, uncontrolled merging junction in Example \ref{example:merge}.

\begin{figure}[t] %%% FIGURE
	\centering
	\begin{subfigure}[b]{0.28\textwidth}
		\includegraphics[width=\textwidth]{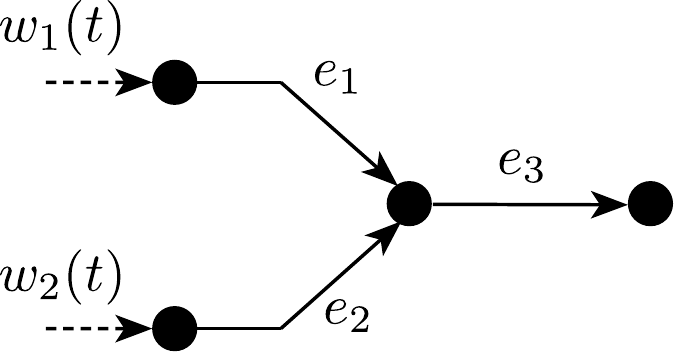}
		\vspace{0.05cm}
		\caption{Simple network with an uncontrolled merging junction.}
		\label{fig:merge_structure}
	\end{subfigure} ~~
	\begin{subfigure}[b]{0.33\textwidth}
		\includegraphics[width=\textwidth]{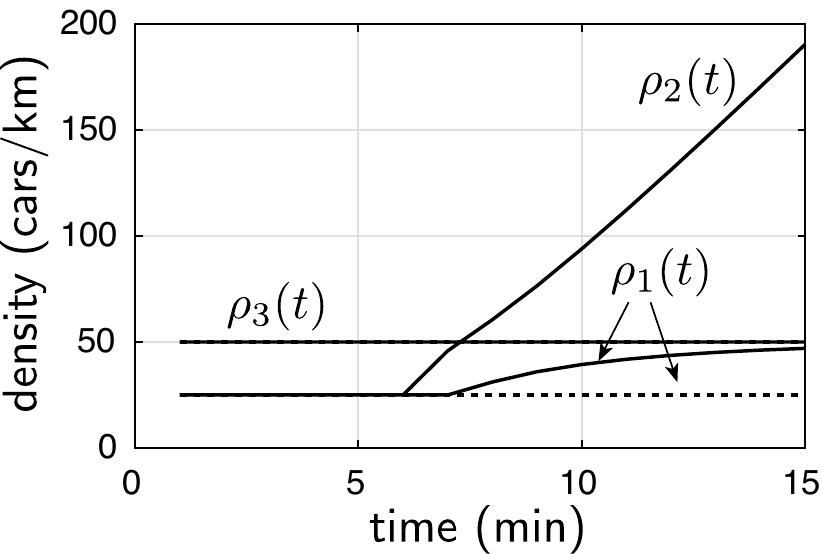}
		\caption{Evolution of the densities.}
		\label{fig:merge_density}
	\end{subfigure} ~~
	\begin{subfigure}[b]{0.33\textwidth}
		\includegraphics[width=\textwidth]{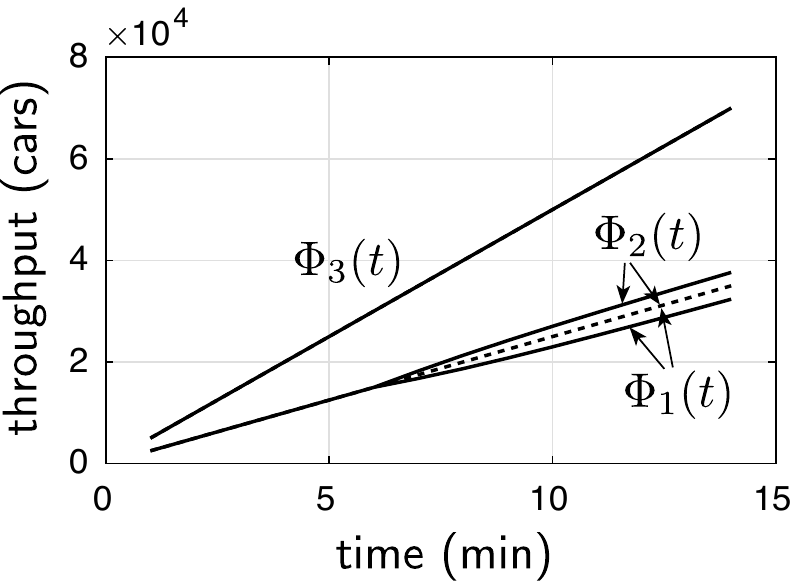}
		\caption{Evolution of the cumulative flows.}
		\label{fig:merge_throughput}
	\end{subfigure}
	\caption{An increase in the external inflow $w_2(t)$ leads to a monotone response in terms of the densities but to a non-monotone response in terms of the cumulative flows. The trajectories obtained if $w_2(t)$ is not increased are shown in dashed.}
	\label{fig:example_merge}
\end{figure} %%% FIGURE
\begin{example}
\label{example:merge}
Consider a simple network as depicted in Figure \ref{fig:example_merge}. We choose again a triangular fundamental diagram with identical parameters among all cells, in particular $l_e = 1$km, $v_e = 100$km/h, $w_e = 25$km/h and $\bar\rho_e = 250$cars/km. The external inflow $w_1(t) = 2500$cars/h is constant, whereas $w_2(t) = 2500$cars/h for $0 \leq t \leq 5$min and $w_2(t) = 5000$cars/h for $t > 5$min; the resulting trajectories are depicted as solid lines in Figures \ref{fig:merge_density} and \ref{fig:merge_throughput}. The case with $w_2(t) = 2500$cars/h for all $t$ serves as a reference; the corresponding trajectories are depicted as dashed lines. We choose the \emph{proportional priority} merging model \cite{kurzhanskiy2010active,coogan2016stability}, in which the supply of free space of the cell downstream of the merging junction is allocated proportionally to the demands of the upstream cells, whenever total demand exceeds supply. In particular,
\begin{equation*}
\phi_1(t) = d_1 \big( \rho_1(t) \big) \cdot \min \left\{ 1,~  \frac{ s_3 \big( \rho_3(t) \big) }{ d_1 \big( \rho_1(t) \big) + d_2 \big( \rho_2(t) \big) } \right\}
\end{equation*}
and
\begin{equation*}
\phi_2(t) = d_2 \big( \rho_2(t) \big) \cdot \min \left\{ 1,~ \frac{ s_3 \big( \rho_3(t) \big) }{ d_1 \big( \rho_1(t) \big) + d_2 \big( \rho_2(t) \big) } \right\} ~.
\end{equation*}
This merging model \emph{violates} Assumption \ref{assumption:merges}, but its dynamics are known to be monotone in the densities \cite{lovisari2014stability,lovisari2014stability2}. Specifically, it can be seen that the increased inflow leads to an increase in the density $\rho_2(t)$, that in turn leads to an increase in $\rho_1(t)$, in accordance with monotonicity of the dynamics in the densities (Figure \ref{fig:merge_density}). By contrast, one observes that in the same scenario, $\Phi_2(t)$ increases faster than its reference trajectory, whereas $\Phi_1(t)$ falls below the reference trajectory (Figure \ref{fig:merge_throughput}). The merging junction is operating at its maximum capacity and an increase in the demand from cell 2 leads to a decrease in the flow from cell 1, which is a non-monotone effect in cumulative flows.
\end{example}

These two examples demonstrate that state-monotonicity of a system might depend on the choice of the state. It is worth highlighting that while monotonicity (respectively mixed-monotonicity) of certain CTM variants in the densities has been used to analyze stability properties \cite{lovisari2014stability,coogan2016stability} and robustness of optimal trajectories \cite{como2016convexity}, the system equations are neither concave nor convex in the densities. Therefore, even though the CTM in densities is monotone for certain networks, in particular for networks without FIFO diverging junctions, Theorem \ref{theorem:csm} cannot be applied directly.

\section{Numerical study} \label{sec:application} %%% SECTION =========================================================== % 

In this section, we consider two freeway networks and verify that solutions of the FNC problem obtained via forward simulation of the solution of the relaxed problem are indeed feasible and optimal. The first example is based on a real freeway in Grenoble, France. Here, we focus on the asymmetric junction as a prototypical model for onramps controlled via ramp metering and demonstrate that onramp occupancy constraints can be included without compromising exactness of the convex relaxation. The second example is based on a fictitious freeway network designed to incorporate both FIFO diverging junctions and all types of merging junction considered in Assumption \ref{assumption:merges}. 

In practice, model uncertainty, non-monotone effects and potential secondary control objectives like fairness constraints also need to be taken into account. Even though these considerations are largely outside the scope of this work, we exemplify in Section \ref{sec:heuristic} how our theoretical results can be used to target the capacity drop in freeway ramp metering, to illustrate how our results can be employed in the design of heuristics for models that violate some of our assumptions.

\subsection{Freeway segment with ramp metering} \label{sec:grenoble} %%% SUBSECTION

\begin{figure}[t] %%% FIGURE
	\centering
		\includegraphics[width=0.45\textwidth]{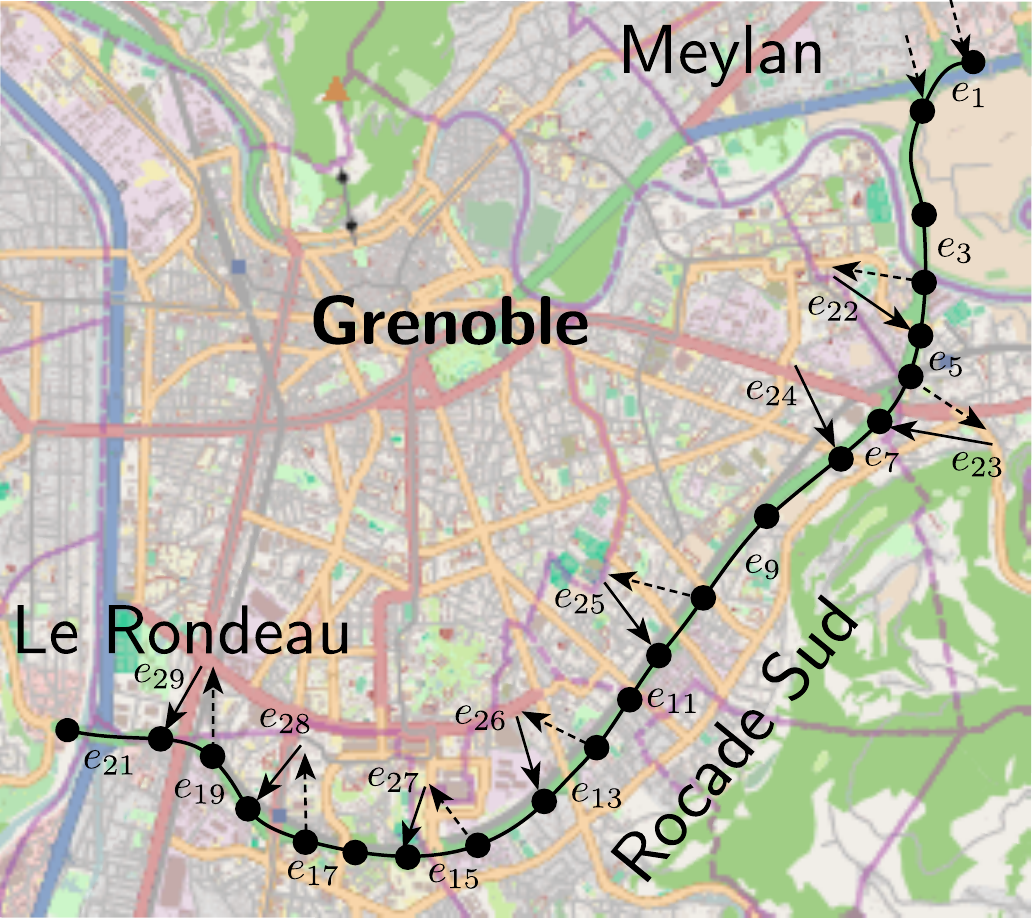}
		\vspace{0.05cm}
		\caption{Road topology of the Rocase Sud. Map data \textcopyright 2016 Open Street Maps.}
		\label{fig:grenoble_structure}
\end{figure}
	
\begin{figure}[t] %%% FIGURE
	\centering	
	\begin{subfigure}[b]{0.4\textwidth}
		\includegraphics[width=\textwidth]{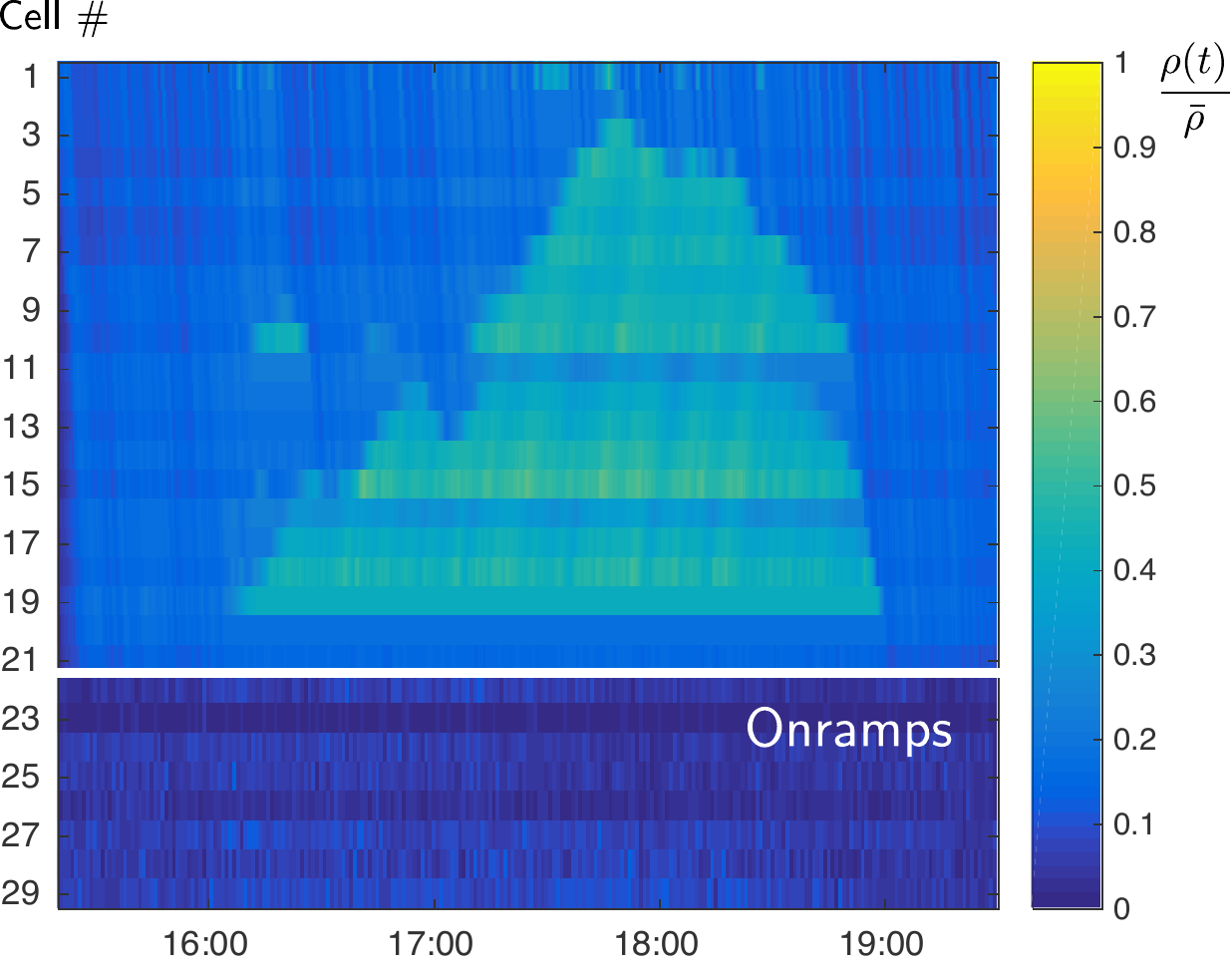}
		\caption{Densities, uncontrolled case.}
		\label{fig:grenoble_1}
	\end{subfigure} 
	\hspace{0.5cm}
	\begin{subfigure}[b]{0.4\textwidth}
		\includegraphics[width=\textwidth]{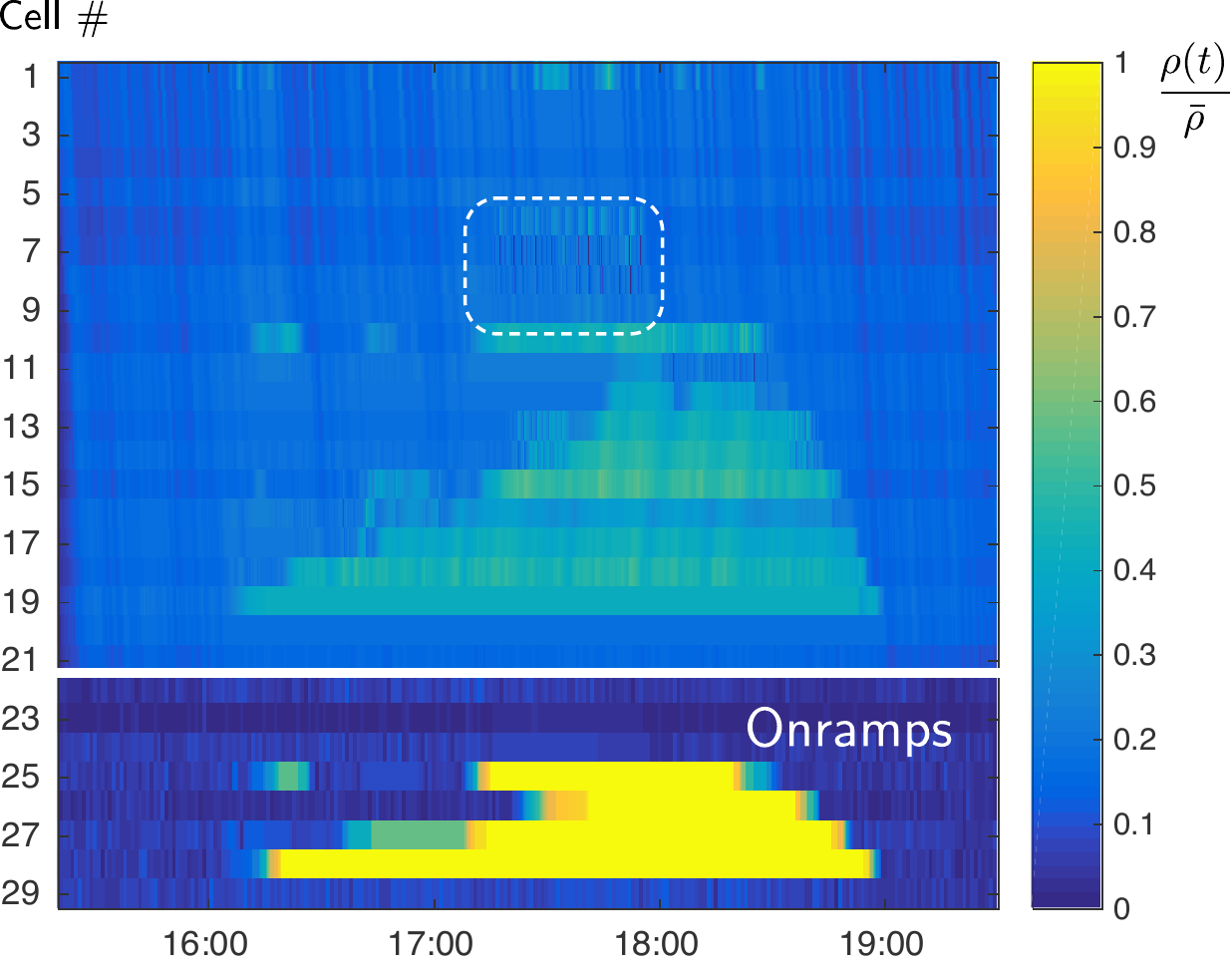}
		\caption{Densities, optimal LP solution.}
		\label{fig:grenoble_2}
	\end{subfigure}
% \vspace*{0.6cm}
	\begin{subfigure}[b]{0.4\textwidth}
		\includegraphics[width=\textwidth]{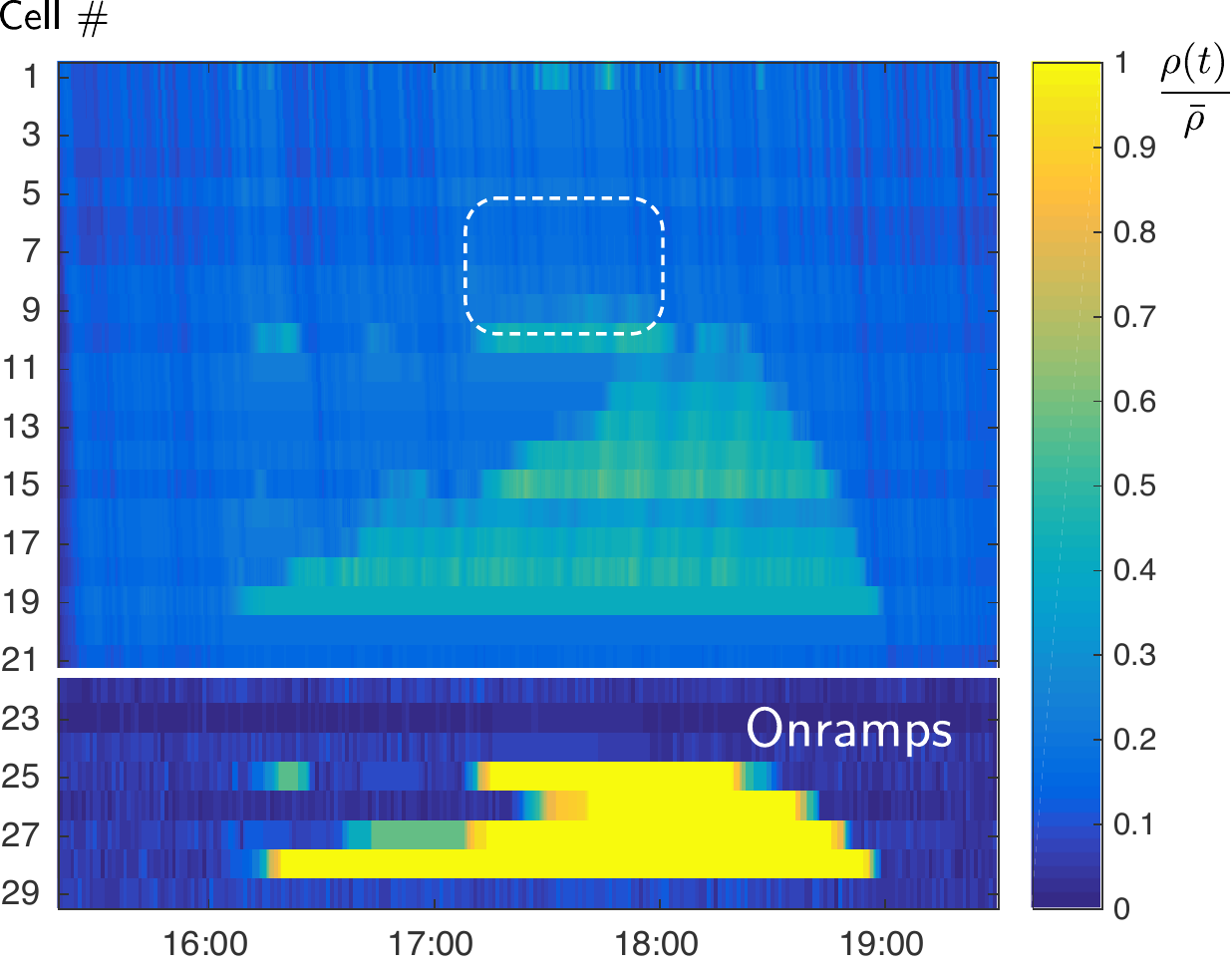}
		\caption{Densities, forward simulation.}
		\label{fig:grenoble_3}
	\end{subfigure} ~
	\begin{subfigure}[b]{0.23\textwidth}
		\includegraphics[width=\textwidth]{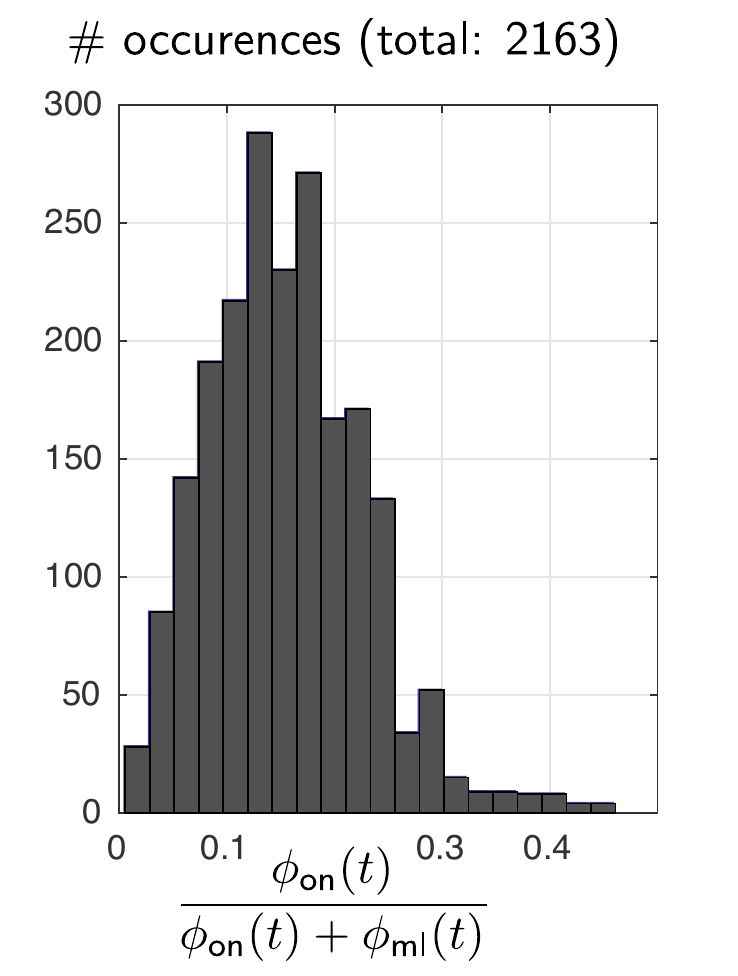}
		\caption{Onramp flow in relation to mainline flow.}
		\label{fig:hist_onramps}
	\end{subfigure} ~
	\begin{subfigure}[b]{0.3\textwidth}
		\includegraphics[width=\textwidth]{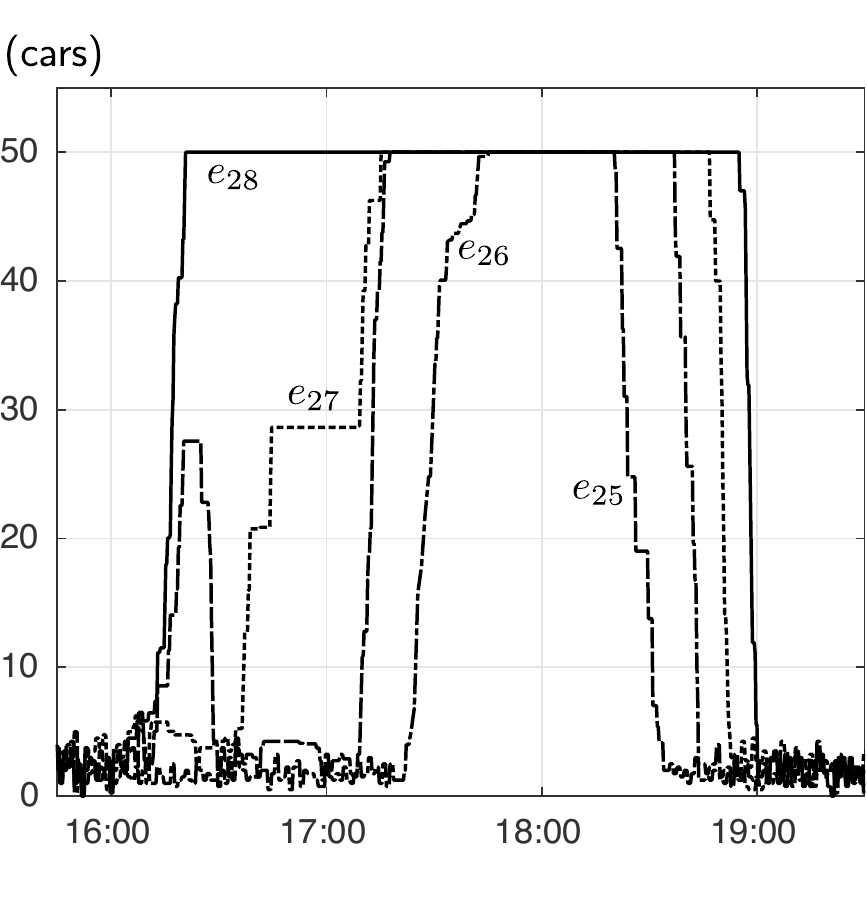}
		\caption{Onramp queues.}
		\label{fig:onramp_queues}
	\end{subfigure}
	\caption{ Simulation results for the Rocade Sud, for the afternoon/evening rush-hour of April 14th, 2014. The densities without ramp metering are depicted in Figure \ref{fig:grenoble_1}. The flows in the optimal solution depicted in Figure \ref{fig:grenoble_2} do not equal the minimum of demand and supply everywhere, in particular during the duration of congestion in cells $e_6$ to $e_9$ (indicated by the white-dashed boxes). The forward simulation (Figure \ref{fig:grenoble_3}) creates a solution that satisfies the system equations everywhere, with the same objective value (TTS). Figure \ref{fig:hist_onramps} shows the distribution of metered onramp inflow divided by total cell inflow, while the mainline is congested. It suggests that modeling the metered onramps as asymmetric junctions is justified for this particular freeway and demand profile, as condition \eqref{eq:asmerge_assumption} is satisfied by a large margin. Figure \ref{fig:onramp_queues} depicts the evolution of relevant onramp queues.}
	\label{fig:grenoble}
\end{figure} %%% FIGURE

In this section, we consider the problem of ramp metering control of the Rocade Sud, a freeway with only onramp and off-ramp junctions. The freeway model in question is based on a congestion-prone freeway in the vicinity of Grenoble \cite{de2015grenoble} with 8 metered onramps and 7 off-ramps. The mainline flow dynamics are modeled using a piecewise-affine fundamental diagram with parameters as described in Appendix \ref{appendix:rocade}. Ramp metering will be installed on this freeway in the near future.

We model the metered onramps using asymmetric junctions. The advantage of the asymmetric junction model over the symmetric one is that only control of the inflow from the onramp is necessary in the former model. In particular, the onramp model using asymmetric junctions will be similar to the one used in the asymmetric cell transmission model (ACTM) \cite{gomes2008behavior}. In the ACTM, it is assumed that inflow from the onramps is \emph{not} constrained by congestion on the mainline, similar to the controlled flow in an asymmetric junction. Consider a cell $e \in \mathcal{S}$ modeling a metered onramp. In the ACTM, the onramp is modeled as an integrator 
\begin{equation*}
\rho_e(t+1) = \rho_e(t) + \frac{\Delta t}{l_e} \cdot \big( w_e(t) - r_e(t) \big) ,
\end{equation*}
with $\rho_e(t)$ the onramp occupancy (here expressed as a density), $w_e(t)$ the external inflow and $r_e(t)$ the metered flow entering the mainline. The dynamics are subject to an onramp capacity constraint $\rho_e(t) \leq \bar\rho_e$ and metering constraints $0 \leq r_e(t) \leq \bar r_e$. We can replicate the metering constraint in our model by defining the onramp demand as $d_e(\rho_e(t)) = \max \left\{ \frac{l_e}{\Delta t} \rho_e(t), \bar r_e \right\}$ and including the ramp in the set of asymmetric junctions $e \in \N_A$. This particular demand function is used solely to reproduce the integrator-like onramp dynamics of the ACTM. Neither should the factor $ \frac{l_e}{\Delta t}$ be interpreted as the free-flow velocity on the onramp, nor should one assume that the density $\rho_e(t)$ is necessarily homogeneous on the onramp. It turns out that the capacity constraints for metered onramps can also be included in the FNC problem:
\begin{corollary} %%% COROLLARY
\label{lemma:onramps}
If capacity constraints $\rho_e(t) \leq \bar\rho_e$ for source cells $e \in \mathcal{S}$ are included in the FNC problem \eqref{eq:problem_original} and if the corresponding constraints are also included in the relaxed FNC problem \eqref{eq:problem_relaxed}, then Theorem \ref{theorem:CCTM} remains valid
\end{corollary} 
\begin{proof} %%% PROOF
If transformed according to the definition of cumulative controlled flows, the source cell capacity constraint can be expressed as 
\begin{align*} 
\big( \Phi_e(t) - W_e(t) \big) \cdot \Delta t + l_e \bar \rho_e &\geq 0 ~,
\end{align*}
for all source cells $e \in \mathcal{S}$ in the CCTM. These constraints are affine and hence concave and it is easy to see that they are also state-monotone. Therefore, they can be included in the generic constraints $g_t( x(t), u(t) ) \geq 0$ (according to Theorem \ref{theorem:csm}) when analyzing the relaxation of the CCTM. Thus, all prior arguments asserting exactness of the relaxation also apply.
\end{proof} %%% END proof 
We include capacity constraints $\rho_e(t) \leq \bar\rho_e$ for all metered onramps in the optimal control problem. Off-ramps of the Rocade Sud are modeled using constant turning rates. The densities on the off-ramps are not part of the model and hence, mainline flow will never be obstructed by the state of the off-ramps. However, outflow via the off-ramps is affected by congestion on the mainline, in accordance with the dynamics of FIFO diverging junctions.

The resulting system model satisfies Assumptions \ref{assumption:graph}, \ref{assumption:fd} and \ref{assumption:merges}. Therefore, Theorem \ref{theorem:CCTM} is applicable. It should be noted that for the special case of a freeway segment with only onramp and off-ramp junctions as considered in this example, \cite{gomes2006optimal} have already proven equivalence of the relaxed FNC problem and the original problem.\footnote{Strictly speaking, \cite{gomes2006optimal} introduce additional technical conditions on the sampling time and the existence of a low-demand ``decongestion period" at the end of the horizon which ensure that the solution to the relaxed problem is always feasible for the original problem. However, these technical assumptions are not critical, as even without them, one can create an optimal solution of the original problem by forward simulation, using the optimal inputs of the relaxation. Also, \cite{gomes2006optimal} only consider the triangular fundamental diagram. The extension to a concave fundamental diagram is straightforward, however.}

We pose the finite horizon optimal control problem using the historical (external) traffic demands of the afternoon/evening rush-hour on April 14th, 2014, with the objective of minimizing TTS. Since we assume a PWA fundamental diagram, the optimization problem \eqref{eq:problem_relaxed} can be reformulated as a linear program (LP). For a sampling time of $15$sec\footnote{By Assumption \ref{assumption:fd}, $\Delta t \leq 20$sec for this freeway. The critical cells have length $l_e = 0.5$km and free-flow velocity $v_e = 90$km/h.}, the resulting LP with a horizon of $5$ hours has $67284$ primal variables and $134512$ constraints. It is solved by Gurobi \cite{gurobi} in $270$sec.\footnote{The solution was found using a 2013 MacBook Pro with 2.3GHz Intel i7 processor (4 cores). Gurobi was interfaced via Matlab.} The results of the optimization are depicted in Figure \ref{fig:grenoble_2}. It turns out that the optimizer is not unique and the particular solution found by the optimization routine does not coincide exactly with the fundamental diagram at all times, in particular during the duration of congestion in cells $e_6$ to $e_9$ (indicated by the white-dashed boxes). No off-ramps are present in the respective parts of the freeway (see Figure \ref{fig:grenoble_structure}) and hence, there is no inherent incentive to maximize flows during times when downstream flow is obstructed by congestion. Therefore, it is necessary to perform a forward simulation using the optimal control inputs computed for the relaxed problem. The results of the forward simulation are depicted in Figure \ref{fig:grenoble_3}. For comparison, we also depict the uncontrolled case (without ramp metering) in Figure \ref{fig:grenoble_1}. The optimizer of the relaxed FNC problem and the solution obtained by forward simulation achieve the same cost, denoted as $\TTS^*$, as predicted by Theorem \ref{theorem:CCTM}. In comparison to the uncontrolled case with cost $\TTS_{ol}$ (``open loop"), an improvement of $\frac{\TTS_{\text{ol}} - \TTS^*}{\TTS_{\text{ol}}} = 3.5\%$ is achieved. Much of the time spent is in fact due to the \emph{free-flow travel time} $\text{FTT}$, that is, the hypothetical time spent if all vehicles travel at free-flow velocity at all times.\footnote{The FFT can be computed by setting $d_e(\rho_e(t)) = v_e \cdot \rho_e(t)$ and assuming cells with infinite capacity, that is, $s_e( \rho_e(t) ) = + \infty$ and $\bar \rho_e = + \infty$.} Naturally, free-flow travel time cannot be reduced by ramp metering. If one considers only the decrease in delay, defined as $\TTS - \text{FTT}$ \cite{gomes2006optimal}, relative savings of $\frac{\TTS_{\text{ol}} - \TTS^*}{\TTS - \text{FTT}} = 12.1\%$ are obtained. In the optimal solution, cars are held back of the onramps if the corresponding part of the mainline is congested, congestion obstructs or threatens to obstruct upstream FIFO junctions (off-ramps) and sufficient space on the metered onramp is available. Figure \ref{fig:onramp_queues} depicts the evolution of four onramp queues. The onramp queues are constrained to $50$ cars each, and it is apparent that these constraints are satisfied. As stated before, the onramp flow is not explicitly constrained by the supply of free space on the mainline in the ACTM. In \cite{gomes2006optimal} an a posteriori check is performed, with the conclusion that flows from onramps can easily be accommodated on the mainline in the optimal solution. Here, we compare the inflows from the onramps $\phi_{\text{on}}(t)$ to the mainline flow $\phi_{\text{ml}}(t)$ at the same location during times of mainline congestion. The ratio $\frac{ \phi_{\text{on}}(t) }{ \phi_{\text{on}}(t) + \phi_{\text{ml}}(t) }$ is depicted in Figure \ref{fig:hist_onramps}. It turns out that the onramp inflow typically accounts for less than one-third of the total flow. Thus, condition \eqref{eq:asmerge_assumption} is satisfied by a large margin, and the asymmetric junction model seems justified for this particular freeway and demand profile.

\subsection{A heuristic targeting the capacity drop} \label{sec:heuristic} %%% SUBSECTION

Recall that non-decreasing, concave demand functions have been assumed to derive Theorem \ref{theorem:CCTM}. In case of a capacity drop as in Figure \ref{fig:ctm_fd_c}, the demand function is not concave and hence, the natural relaxation \eqref{eq:problem_relaxed} of the FNC problem is no longer convex. In such a case, it is also easy to construct examples where control of merging flows alone is not sufficient to make optimal solutions to the relaxation feasible in the non-relaxed problem. This suggests that the optimization of models involving a capacity drop is substantially more difficult, as no exact convex relaxation is available.

In this section, we seek to demonstrate that instead of invoking non-convex optimization, one can also use the results in this work to target the effects of the capacity drop heuristically, while retaining a convex problem.
To do so, we consider again the Rocade Sud as in the previous section, but we now assume that every demand function of mainline cells exhibits a capacity drop as depicted in Figure \ref{fig:ctm_fd_c}, where the demand function decreases by $10\%$ as soon as a critical density is exceeded. In addition, the maximal throughput of all cells is increased so that the throughput in congestion, when the capacity drop is active, equals the throughput in the monotone model, so that comparable congestion patterns result. All other model parameters remain unchanged. 

In times of congestion, the relaxed FNC problem often allows for multiple optimal solutions with equal objective values. They correspond to situations where the bottleneck flows are equal (in the monotone controller model), but the waiting vehicles are distributed differently on onramp queues and mainline congestion. A capacity drop penalizes mainline congestion. Therefore, we consider a heuristic control objective
\begin{equation*}
\TTS_{\epsilon} := \sum_{t = 1}^{T} \Bigg( \sum_{e \in \E \setminus \N_A} l_e \cdot \rho_e(t) + (1 - \epsilon) \cdot  \sum_{e \in \N_A}  l_e \cdot \rho_e(t) \Bigg) ,
\end{equation*}
where waiting times incurred on the onramps are penalized less than time spent on the mainline. The idea is to choose $\epsilon > 0$ small, to stay close to the objective of minimizing TTS, while selecting a solution in which vehicles are released as late as possible from the onramps. Before using the heuristic objective in simulations, we need to verify that it is state-monotone in the CCTM. To do so, consider
\begin{align*}
c'( \Phi(t) ) &:= \sum_{t=1}^{T} \sum_{e \in \E} \hat c_e \Phi_e(t) + \epsilon \cdot \sum_{t=1}^{T} \sum_{e \in \N_A}  \big( \underbrace{ W_e(t) + \hat \Phi_e(t) }_{= l_e \rho_e(t)}  \big) = C_{\text{W}} - \TTS_\epsilon, 
\end{align*}
with $\hat c_e$ as defined in Section \ref{sec:CCTM}. Maximization of $c'( \Phi )$ encodes minimization of $\TTS_\epsilon$ and $\hat c_e \geq 0$ and $\epsilon > 0$ imply that $c'( \Phi(t) )$ is state monotone indeed. Note that we made use of the auxiliary states $\hat \Phi(t)$ once more to achieve state monotonicity.

\begin{figure}[t] %%% FIGURE
	\centering	
	\begin{subfigure}[b]{0.44\textwidth}
		\includegraphics[width=\textwidth]{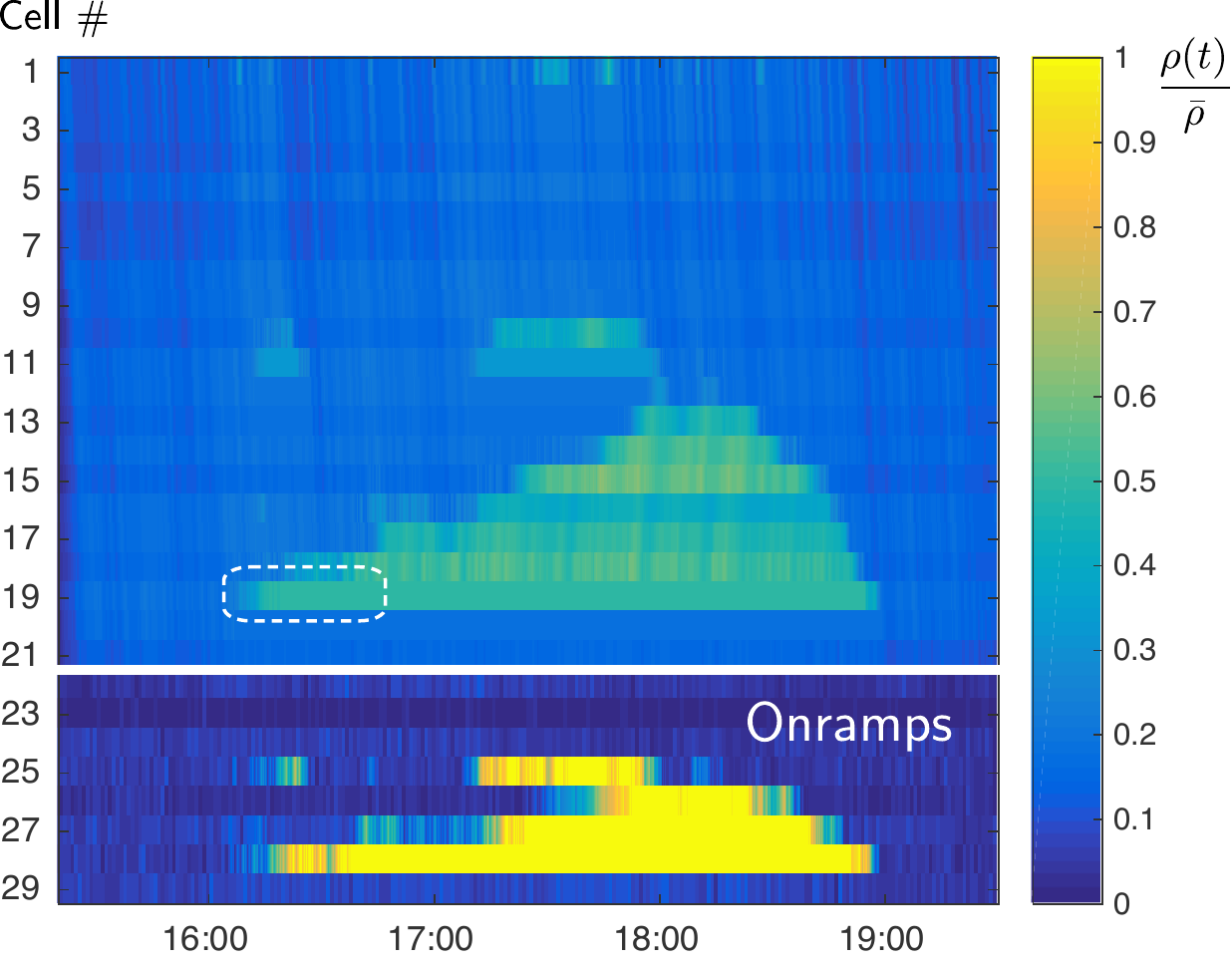}
		\caption{Densities, receding horizon control with default objective ($\TTS$).}
		\label{fig:grenoble_5}
	\end{subfigure}
	\hspace{0.5cm}
	\begin{subfigure}[b]{0.44\textwidth}
		\includegraphics[width=\textwidth]{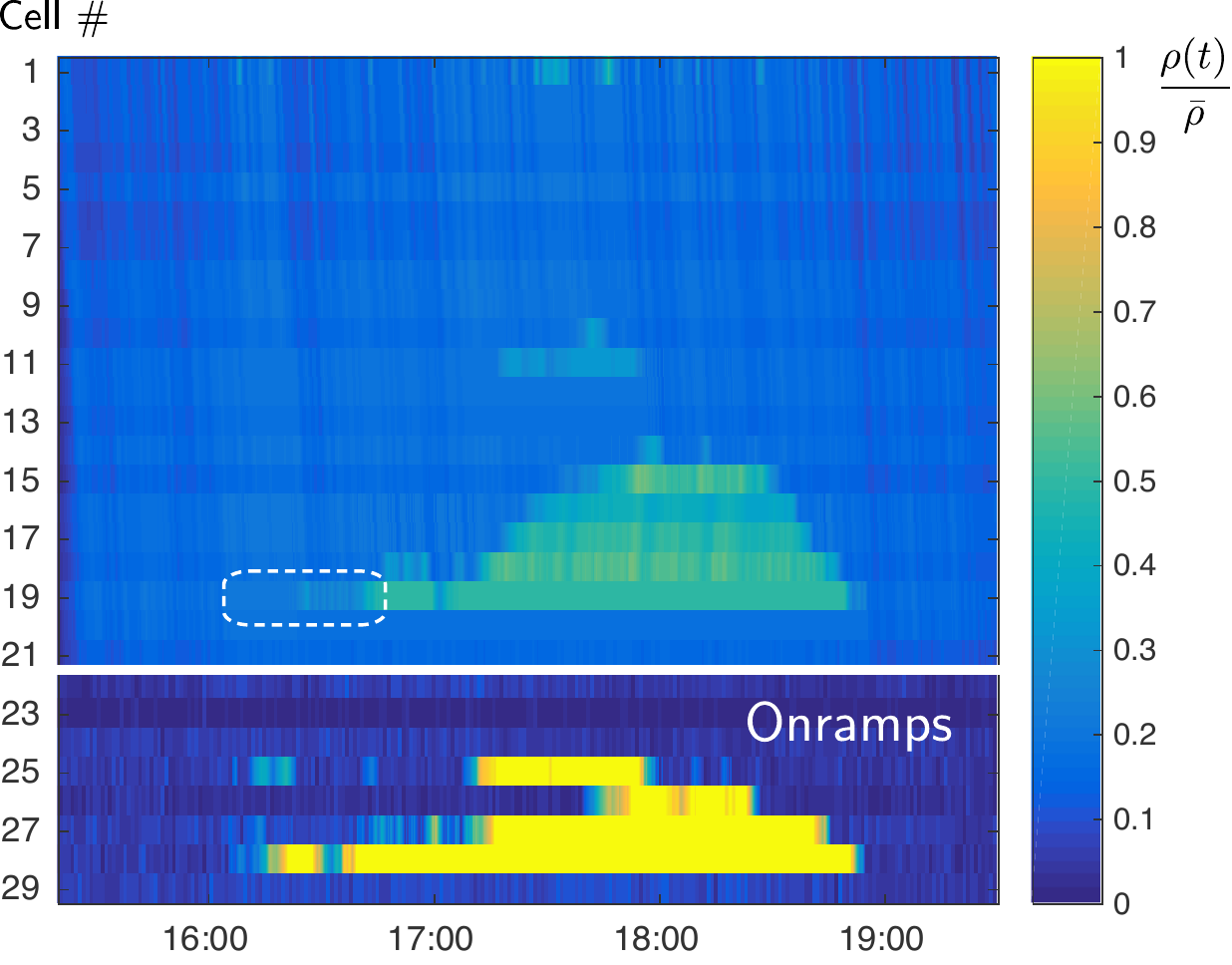}
		\caption{Densities, receding horizon control with heuristic objective ($\TTS_\epsilon$).}
		\label{fig:grenoble_6}
	\end{subfigure} ~
	\caption{ Simulation results for the Rocade Sud using the alternative model with capacity drop, for the afternoon/evening rush-hour of April 14th, 2014. The heuristic objective is designed to favor queues on the onramps over congestion on the mainline. In this particular example, the onset of congestion in cell $e_{19}$ is delayed by the heuristic, which delays the onset of the capacity drop and achieves higher throughput in the period indicated by the white-dashed box.  }
	\label{fig:grenoble_nm}
\end{figure} %%% FIGURE

We now compare the performance of receding horizon controllers based on either objective function. The capacity drop cannot be included in the controller model if monotonicity is to be retained. Instead, a non-decreasing demand function is assumed by the controller, with maximal value equal to the average of free-flow throughput and throughput in congestion.\footnote{ A possible alternative choice, to set throughput in the controller model to the maximal, free-flow throughput, leads to even better performance of the policy using the heuristic objective in simulation. The resulting trajectory seems unrealistic, however, since in practice, model uncertainty prevents any controller from stabilizing the flow exactly at the critical density, where maximal, free-flow throughput is achieved. Choosing the maximal throughput in the controller model as the average of free-flow throughput and throughput in congestion acts as a safety margin. } Due to the model mismatch between system and controller, we consider a receding horizon implementation, where an optimization problem with horizon $10$min is solved every $2$min and the control inputs of the first $2$min of the prediction horizon are applied to the system. Density contours of the results are depicted in Figure \ref{fig:grenoble_nm}. The receding horizon controller using the default objective achieves an improvement of $2.8$\% in TTS ($10.1\%$ in delay), while the controller using the heuristic objective with $\epsilon = 0.1$ achieves an improvement of $9.68\%$ in TTS ($35.6\%$ in delay). This improvement in performance is due to the heuristic objective initially holding back more traffic on the onramps, thereby delaying the onset of congestion. This allows a larger flow out of cell $e_{20}$ for the time period indicated in Figure \ref{fig:grenoble_nm}, since the capacity drop is avoided temporarily. Note that smaller throughput before the onset of congestion leads to longer congestion queues throughout the congestion, and hence to a large increase in total waiting times. In this example, the capacity drop influences the effectiveness of ramp metering significantly, suggesting that it should not be ignored in the design of ramp metering control policies. However, the results also demonstrate that the proposed heuristic is able to target the effects of the capacity drop and improve performance, while employing a monotone system model in the optimization which does not include the capacity drop explicitly. While this particular heuristic is tailored towards ramp metering and does not generalize to e.g.\ symmetric junctions in an obvious manner, we believe it exemplifies how the main result of this work can also be used as a guideline in the design of ``efficient" heuristics that target non-monotone effects.

\subsection{General network with controlled merging junctions} \label{sec:ex_network} %%% SUBSECTION

\begin{figure}[t] %%% FIGURE
	\centering
		\includegraphics[width=\textwidth]{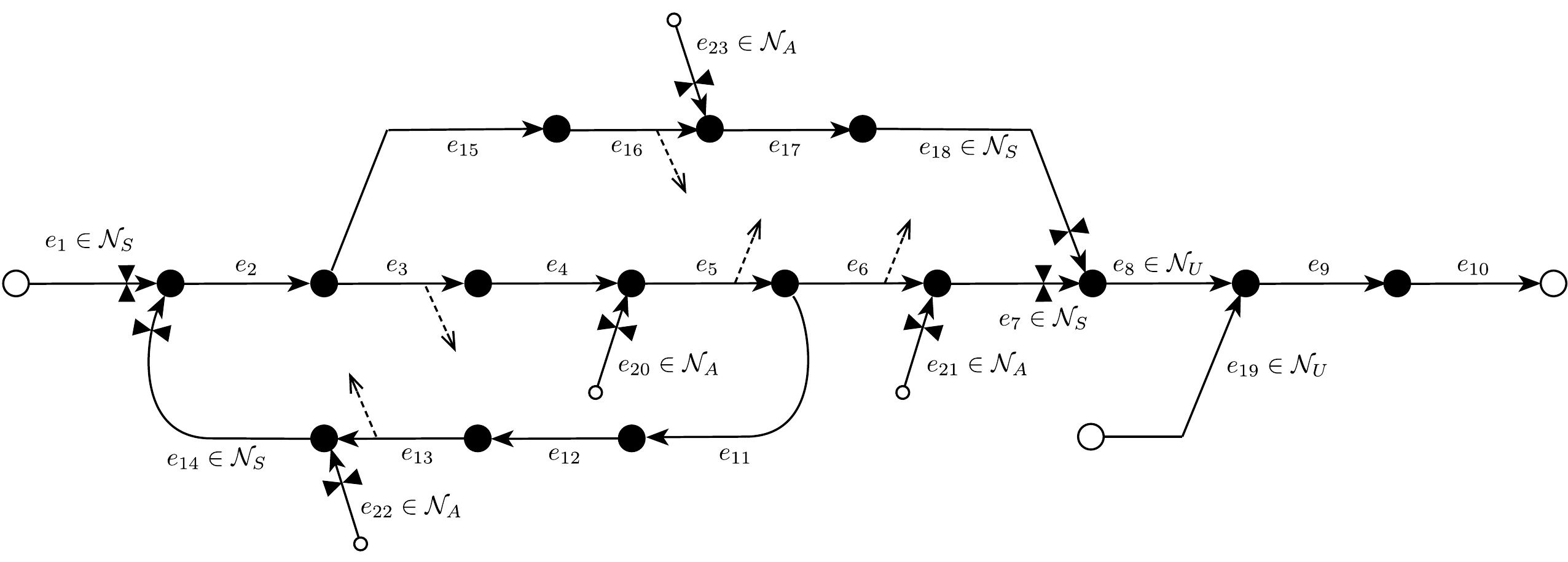}
		\caption{Artificial network topology with 23 cells, two symmetric junctions, four asymmetric junctions modeling metered onramps, one sub-critical junction, two mainline FIFO diverging junctions and five FIFO offramp junctions.}
	\label{fig:ex_network}
\end{figure} %%% FIGURE

\begin{figure}[t] %%% FIGURE
	\centering
	\begin{subfigure}[b]{0.32\textwidth}
		\includegraphics[width=\textwidth]{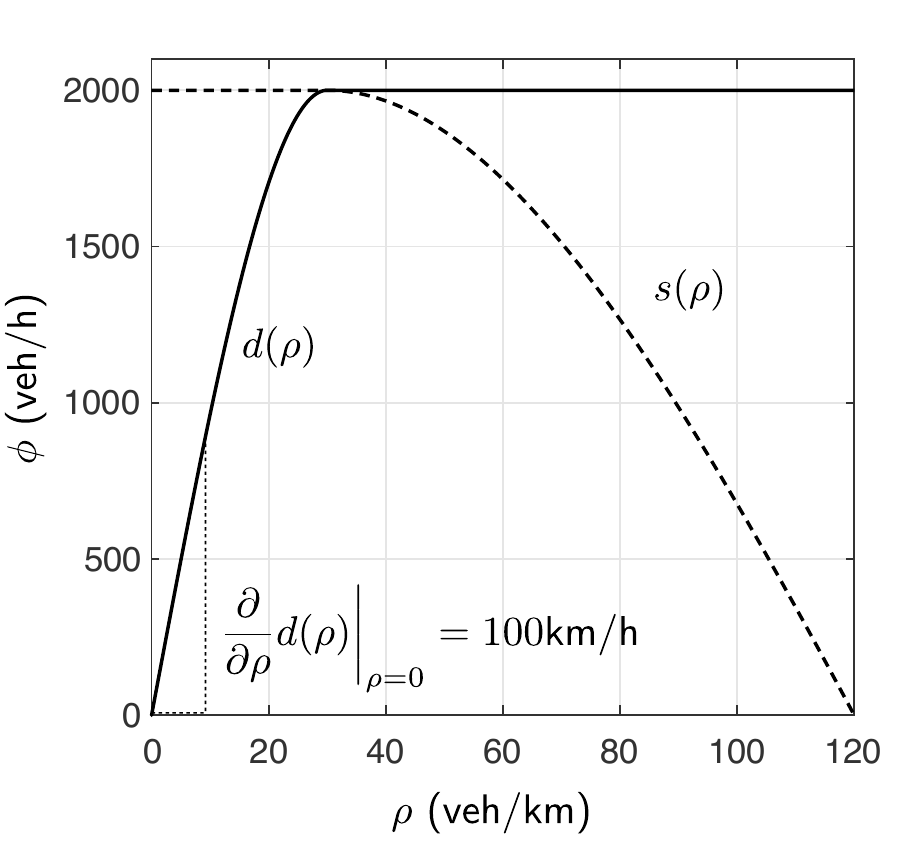}
		\caption{Concave demand and supply functions of a single lane, modeled using 3$^\text{rd}$-order polynomials.}
		\label{fig:nonlinear_fd}
	\end{subfigure} ~ 
	\begin{subfigure}[b]{0.32\textwidth}
		\includegraphics[width=\textwidth]{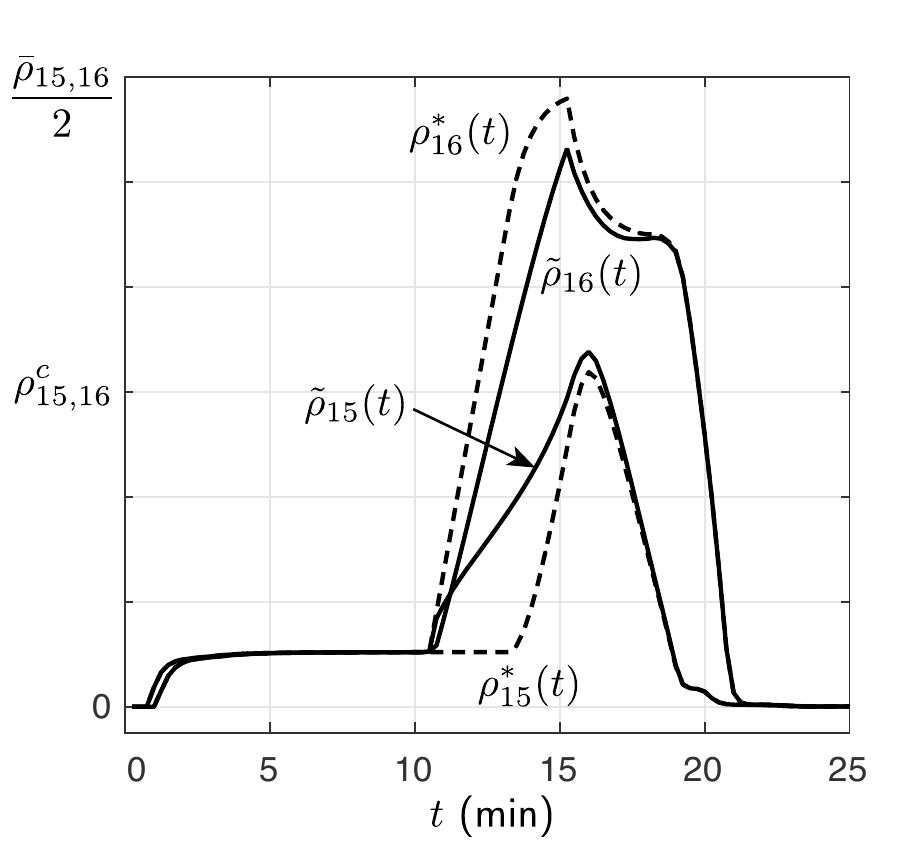}
		\caption{Solution of the relaxed problem $\tilde \rho_{15,16}(t)$ in comparison to the result of the forward simulation $\rho^*_{15,16}(t)$.}
		\label{fig:cell_15_16}
	\end{subfigure} ~
		\begin{subfigure}[b]{0.32\textwidth}
		\includegraphics[width=\textwidth]{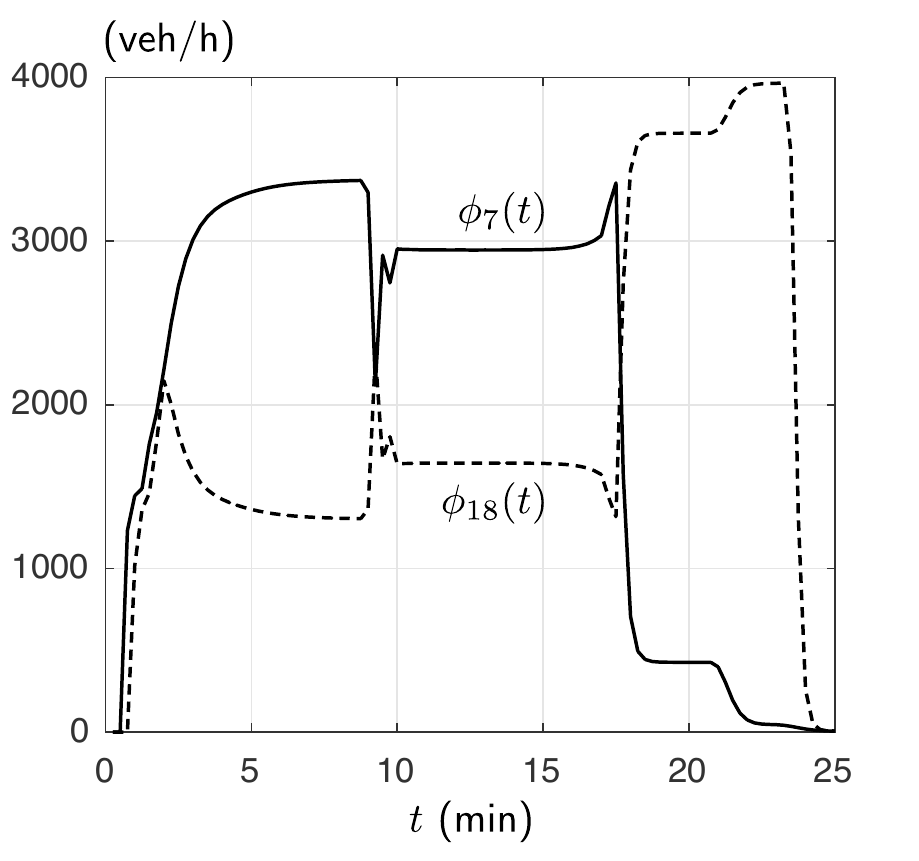}
		\caption{Controlled flows in the symmetric junction, where $\phi_{7}(t)$ and $\phi{18}(t)$ merge.}
		\label{fig:symmetric_flow}
	\end{subfigure} 
	
	\vspace{0.3cm}
	
	\begin{subfigure}[b]{0.48\textwidth}
		\includegraphics[width=\textwidth]{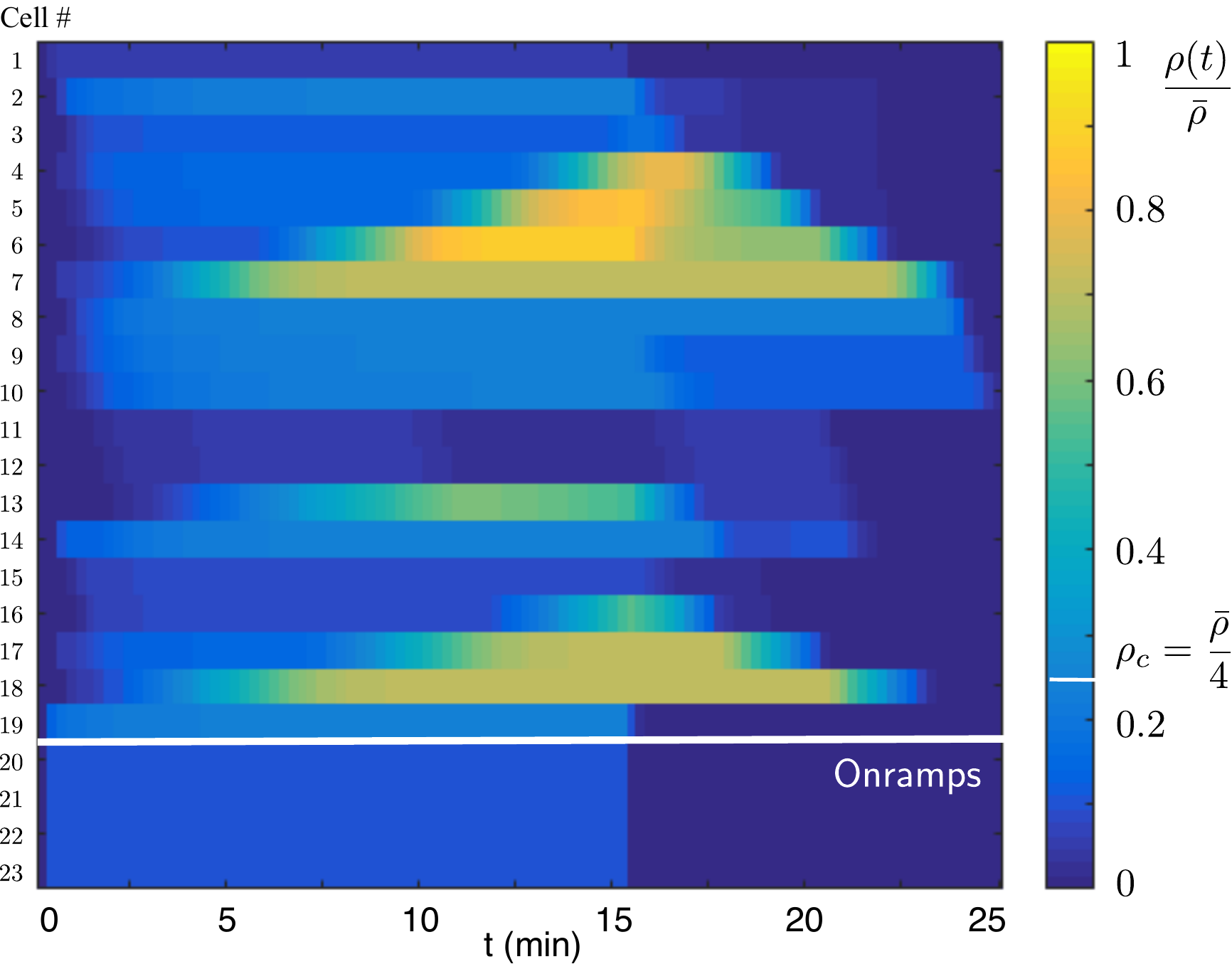}
		\caption{Densities without demand control, using the proportional-priority merging model.}
		\label{fig:network_ol}
	\end{subfigure}  ~ ~ 
	\begin{subfigure}[b]{0.48\textwidth}
		\includegraphics[width=\textwidth]{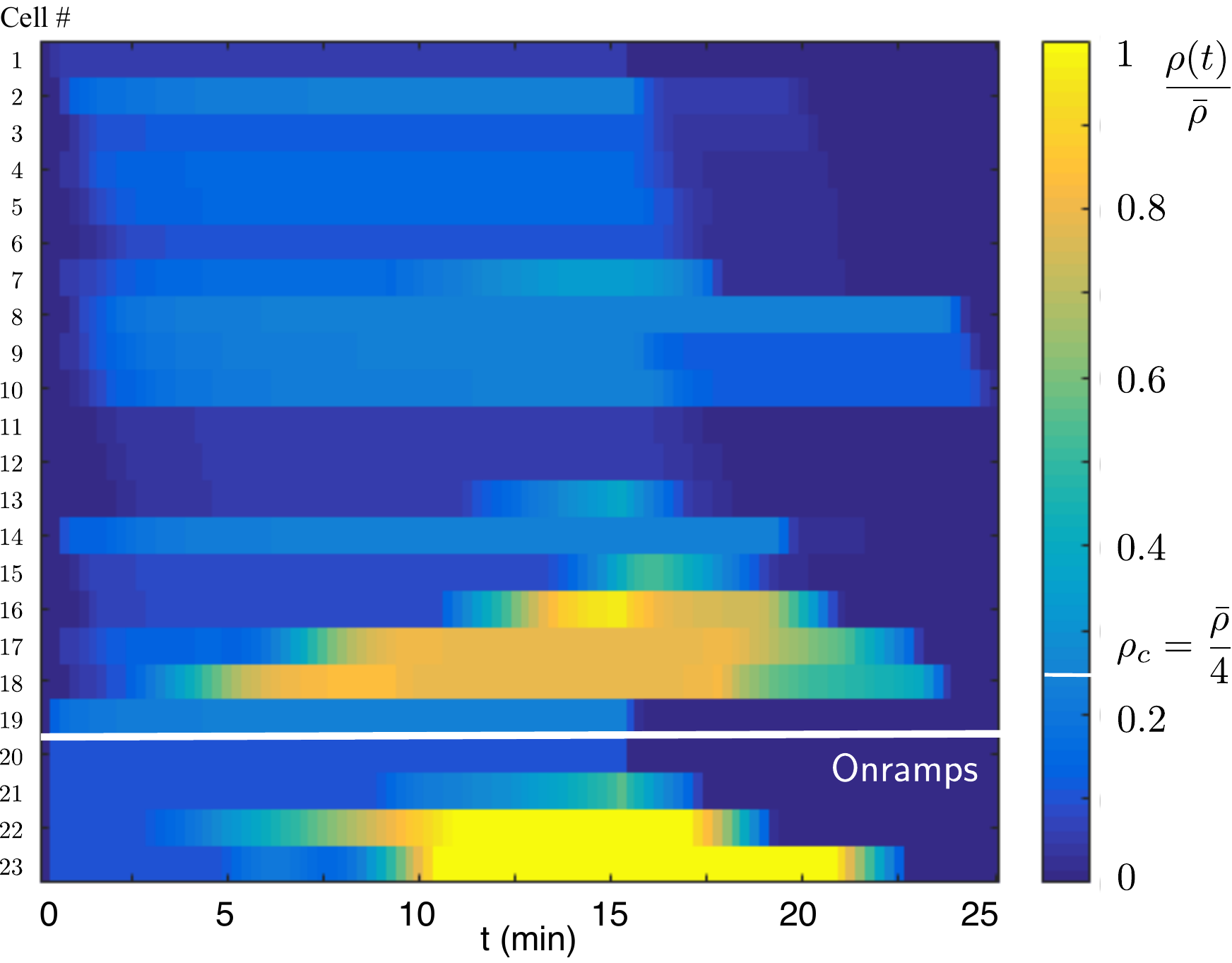}
		\caption{Densities obtained by forward simulation using optimal control inputs.}
		\label{fig:network_cl}
	\end{subfigure}
	\caption{Simulation results for the network depicted in Figure \ref{fig:ex_network}. The main bottleneck is the merging junction in cell $8$. The optimal control alleviates the congestion forming from this bottleneck, mainly by ramp metering using ramps $21$ and $23$. In addition, preference is given to flows from cell $7$ (as opposed to cell $18$, see Figure \ref{fig:symmetric_flow}), where congestion threatens to block several upstream FIFO diverging junctions. }
	\label{fig:simulation_grenoble}
\end{figure} %%% FIGURE

In the following, we consider the freeway network depicted in Figure \ref{fig:ex_network}. We aim to minimize TTS by using ramp metering and mainline demand control for merging junctions. For ramp metering, we assume that the onramp model described in the previous section is used. For merging junctions which are not onramp junctions, we use the symmetric junction model. We first establish that control of merging flows in such a symmetric junction can be realized via demand control. To this end, we follow the notation of \cite{como2016convexity} and introduce demand control inputs $\alpha_e(t) \in [0,1]$ for cells $e \in \N_S \cup \N_A$. The controlled demand of such a cell is defined as
\begin{equation*}
d_e^c \big( \rho_e(t) \big) = \alpha_e(t) \cdot \de .
\end{equation*}
The controlled demand replaces the uncontrolled demand $d_e( \rho_e(t) )$ in the system equations. Now, assume a desired, controlled flow $\phi^*_e(t)$ has been computed by solving the relaxed FNC problem \eqref{eq:problem_relaxed}. Then, set $\alpha_e(t) := \frac{\phi_e^*(t)}{d_e ( \rho_e(t) )}$. Since constraints in the relaxed FNC problem ensure that controlled flows never exceeds the supply of free space, the merging junction is always in free-flow and the desired controlled flows are achieved. This result holds for both the proportional-priority merging model \cite{kurzhanskiy2010active,coogan2016stability} and Daganzo's priority rule \cite{daganzo1995cell}, since the models only differ in congestion. In \cite{muralidharan2012optimal}, variable speed limits are suggested for demand control. In particular, the controlled demand of a cell with variable speed limit $v_e(t)$ is given as
\begin{align*}
d^c_e \big( \rho_e(t) , v_e(t) \big) = \min \big( \de, v_e(t) \cdot \rho_e(t) \big) ~,
\end{align*}
assuming perfect compliance with the speed limit. By choosing $v^*_e(t) = \frac{\phi^*_e(t)}{\rho_e(t)}$, flow control can be realized. \\

In the network depicted in Figure \ref{fig:ex_network}, two mainline junctions are modeled as symmetric junctions. We assume that demand control as outlined above is used in the corresponding cells $e \in \N_S = \{ e_1, e_7, e_{14}, e_{18} \}$. Four onramps ($e_{20}$, $e_{21}$, $e_{22}$ and $e_{23}$) are present; we assume that they are all used for ramp metering. We model the onramps as asymmetric junctions and use the onramp model presented in the previous section, which replicates a simple integrator behavior for onramps with limited capacity. Condition \eqref{eq:asmerge_assumption}, that onramp demand can always be served, is checked a posteriori. One merging junction is modeled as a sub-critical junction ($e_8, e_{19} \in \N_U$). The network contains two FIFO diverging junctions (downstream of $e_2$ and $e_5$) within the network in addition to several off-ramp junctions, which can be interpreted as special cases of FIFO diverging junctions where the off-ramp is never congested. The corresponding turning rates are $\beta_{3,2} = \frac{2}{3}$, $\beta_{15,2} = \frac{1}{3}$, $\beta_{4,3} = \frac{4}{5}$, $\beta_{6,5} = \frac{1}{2}$, $\beta_{11,5} = \frac{1}{4}$, $\beta_{7,6} = \frac{4}{5}$, $\beta_{8,7} = \frac{4}{5}$, $\beta_{14,13} = \frac{4}{5}$ and $\beta_{17,16} = \frac{4}{5}$. For mainline cells (all cells except the onramps) we assume that every lane is described by the same fundamental diagram. To demonstrate the applicability of our result to concave demand and supply functions, we model the demand function of a single lane by a third-order polynomial $d(\rho) = d_3 \rho^3 + d_2 \rho^2 + d_1 \rho + d_0$ for $0 \leq \rho \leq \rho^c$, where the coefficients are chosen such that $d(0) = 0$cars/h, $\frac{d}{d\rho} d( 0 ) = 100$km/h, $d(\rho^c) = 2000$cars/h and $\frac{d}{d \rho} d( \rho^c ) = 0$km/h. Here, $\rho^c$ denotes the critical density, which is chosen as $\rho^c = 30$cars/km per lane. For $\rho^c \leq \rho \leq \bar\rho$, the demand function is constant. Similarly, the supply function is modeled by a third-order polynomial $s(\rho) = s_3 \rho^3 + s_2 \rho^2 + s_1 \rho + s_0$ for $\rho^c \leq \rho \leq \bar\rho$, where the coefficients are chosen such that $s(\rho^c) = 2000$cars/h, $\frac{d}{d \rho} s( \rho^c) = 0$km/h, $s(\bar\rho) = 0$ and $\frac{d}{d \rho} s( \bar\rho ) = -35$km/h. For $0 \leq \rho \leq \rho^c$, the supply function is constant. The resulting demand and supply functions are depicted in Figure \ref{fig:nonlinear_fd}. One can numerically verify that it satisfies Assumption \ref{assumption:fd}. The number of lanes differs between cells. We assume cells $e_1$, $e_2$, $e_3$, $e_4$, $e_5$, $e_9$ and $e_{10}$ are comprised of three lanes, cell $e_{14}$ of only one lane and the remainder of the mainline cells of two lanes. All cells are $0.5$km long and the sampling time is chosen as $15$sec to satisfy Assumption \ref{assumption:fd}. The network structure and the parameters are chosen to exemplify the behavior of FIFO diverging junctions and symmetric junctions. 

We chose a simple piecewise-constant demand pattern in order to facilitate visual analysis of the results. Specifically we set $w_1(t) = 2000$cars/h, $w_{19}(t) = 1000$cars/h and $w_{20}(t) = w_{21}(t) = w_{22}(t) = w_{23}(t) = 750$cars/h for $1\leq t \leq 60$ respectively, and zero otherwise. Simulation results without control are depicted in Figure \ref{fig:network_ol}. In this simulation, we use the proportional-priority  merging rule as in Example \ref{example:merge} for the symmetric junctions and do not restrict the flow from onramps onto the mainline. The merging junction upstream of cell $e_8$ is the major bottleneck of this network, causing congestion to form in cells $e_7$ to $e_4$ and in cells $e_{18}$ to $e_{16}$. In addition, there is a minor bottleneck at cell $e_{14}$. We use the TTS as the performance metric and obtain a total cost of $\TTS_{\text{ol}} = 222.5$h.

Next, we solve the relaxed FNC problem \eqref{eq:problem_relaxed}. For the given network and a horizon of $T = 25$min, the resulting problem is a nonlinear, convex optimization problem with $5543$ primal variables and $8280$ constraints (in addition to lower and upper bounds on the primal variables). It is solved to optimality by IPOPT \cite{wachter2006implementation} using the L-BFGS algorithm in $34$sec\footnote{IPOPT was chosen due to the nonlinear constraints resulting from the nonlinear fundamental diagram. Recall that despite the nonlinearities, the relaxed problem is still convex. The solution was found using a 2013 MacBook Pro with 2.3GHz Intel i7 processor (4 cores). IPOPT was interfaced via Matlab.}. The optimizer is not unique and some cells of the optimizer found by IPOPT do not satisfy the non-relaxed fundamental diagram -- that is, some flows are strictly lower than the corresponding minimum of traffic demand and supply. Therefore, it is necessary to perform a forward simulation using the optimal control actions of the relaxed problem. The results of the forward simulation are depicted in Figure \ref{fig:network_cl}. The cost achieved in this simulation is $\TTS^* = 201.1$h, equal to the optimal cost of the relaxed FNC problem as predicted by Theorem \ref{theorem:CCTM}. In this case, differences between the particular solution of the relaxed problem and the results of the forward simulation occur for example in cells $e_{15}$ and $e_{16}$, as depicted in Figure \ref{fig:cell_15_16}. No off-ramp is present in cell $e_{15}$ and therefore, there is no reason to maximize the flow from cell $e_{15}$ to cell $e_{16}$ as long as cell $e_{16}$ is congested. 

The different objective values for the uncontrolled and the controlled case show an improvement of $\frac{\TTS_{ol} - \TTS^*}{\TTS_{ol}} = 9.6\%$. The relative savings in terms of delay are $\frac{ \TTS_{\text{ol}} - \TTS^*}{ \TTS_{\text{ol}} - \text{FTT}} = 16.9\%$, for a free-flow speed of $100$km/h on an empty freeway (Figure \ref{fig:nonlinear_fd}). The difference can be attributed to the elimination of congestion in cells $e_4$ to $e_7$ in the controlled case, as depicted in Figures \ref{fig:network_ol} and \ref{fig:network_cl}. This is mainly achieved by prioritizing the flow from cell $e_7$ over the flow from $e_{18}$ in the downstream merging junction, as depicted in Figure \ref{fig:symmetric_flow}. The congestion in the uncontrolled case restricts off-ramp flows in cells $e_5$ and $e_6$, as well as the flows to cell $e_{11}$ downstream of the FIFO diverging junction. In addition, ramps $e_{22}$ and $e_{23}$ (and to a limited extent, ramp $e_{21}$) engage in ramp metering (see Figure \ref{fig:network_cl}) to prevent or reduce local congestion, such that the flows from the respective upstream off-ramps are not obstructed.

\section{Conclusions} \label{sec:conclusions}  %%% SECTION =========================================================== %

We have demonstrated that the CTM with controlled merging junctions can be equivalently represented as a concave and state-monotone system. Interestingly, the FIFO model for diverging junctions is monotone in this alternative representation, even though its dynamics are not state-monotone in the densities. Concavity and state monotonicity of the dynamics have been used to derive an exact, convex relaxation that allows for the efficient solution of finite-horizon optimal control problems. Thereby, we generalize existing results which have shown that the ``natural" relaxation of certain FNC problems can be used to compute an optimal solution of the non-relaxed problem, but under more restrictive assumptions on the model. The main result in this work is based on a novel characterization of the system dynamics, which allows to generalize the results to modifications of the system dynamics as long as these modifications preserve concavity and state monotonicity in terms of the cumulative flows. For example, onramp capacity constraints were added to the model in Section \ref{sec:grenoble}. It was sufficient to verify that the additional constraints are concave and state-monotone in the cumulative flows to extend the proof of exactness of the convex relaxation of the FNC to this case. Unfortunately, we have also seen that the arguably most relevant extension, the case of uncontrolled, proportional-priority merging junctions, does not satisfy these properties.

Future research might focus on such uncontrolled merging junctions nevertheless. In particular, it seems reasonable to ask whether it is possible to find ``partial" relaxations of the FNC problem with only ``few" nonconvex constraints relating to the uncontrolled merging junctions that are tractable to solve numerically, but can be used to compute a solution of the original problem efficiently.
Ongoing work also considers the impact of model uncertainty, in particular uncertainty in the fundamental diagram and in predictions of future, external demand. Even in the presence of bounded uncertainty sets, the worst-case uncertainty realization is hard to determine for traffic networks that are not monotone or admit a monotone reformulation \cite{kurzhanskiy2012guaranteed}. However, preliminary results suggest that networks with controlled merging junctions admit a reformulation which allows to determine the worst-case uncertainty realization, thereby making certain robust optimal control problems tractable.
In addition, only few examples of convex-monotone systems have been reported in the literature. The results in this work raise 
the possibility that there exist other relevant systems which can be transformed into an equivalent, convex-monotone (or concave, state-monotone) form.

\section*{Acknowledgements}
Research was partially supported by the European Union 7th Framework Programme ``Scalable Proactive Event-Driven Decision-making (SPEEDD)" (FP7-ICT 619435). We are grateful to the Grenoble Traffic Lab at INRIA, Grenoble, France for providing the traffic data used in the numerical study presented in Section \ref{sec:grenoble}.

%% BIBLIOGRAPHY %%%%%%%%%%%%%%%%%%%%%%%%%%%%%%%%%%%%%%%%%%%%%%%%%%%%%%%%%%
\FloatBarrier
% \section*{References}
% \bibliography{trb2_bibliography}
\bibliography{/Users/mschmitt/Documents/Docs_Latex/MS_bibliography}

\begin{thebibliography}{10}

\bibitem{angeli2003monotone}
David Angeli and Eduardo~D Sontag.
\newblock Monotone control systems.
\newblock {\em IEEE Transactions on Automatic Control}, 48(10):1684--1698,
  2003.

\bibitem{boyd2004convex}
Stephen~Poythress Boyd and Lieven Vandenberghe.
\newblock {\em Convex optimization}.
\newblock Cambridge University Press, 2004.

\bibitem{de2015grenoble}
Carlos Canudas~de Wit, Fabio Morbidi, Luis~Leon Ojeda, Alain~Y. Kibangou, Iker
  Bellicot, and Pascal Bellemain.
\newblock Grenoble traffic lab: An experimental platform for advanced traffic
  monitoring and forecasting.
\newblock {\em IEEE Control Systems}, 35(3):23--39, 2015.

\bibitem{como2015throughput}
Giacomo Como, Enrico Lovisari, and Ketan Savla.
\newblock Throughput optimality and overload behavior of dynamical flow
  networks under monotone distributed routing.
\newblock {\em IEEE Transactions on Control of Network Systems}, 2(1):57--67,
  2015.

\bibitem{como2016convexity}
Giacomo Como, Enrico Lovisari, and Ketan Savla.
\newblock Convexity and robustness of dynamic network traffic assignment for
  control of freeway networks.
\newblock {\em Transportation Research. Part B: Methodological}, 91:446--465,
  2016.

\bibitem{como2013robust2}
Giacomo Como, Ketan Savla, Daron Acemoglu, Munther~A Dahleh, and Emilio
  Frazzoli.
\newblock Robust distributed routing in dynamical networks---part i: Locally
  responsive policies and weak resilience.
\newblock {\em IEEE Transactions on Automatic Control}, 58(2):317--332, 2013.

\bibitem{como2013robust1}
Giacomo Como, Ketan Savla, Daron Acemoglu, Munther~A Dahleh, and Emilio
  Frazzoli.
\newblock Robust distributed routing in dynamical networks--part ii: Strong
  resilience, equilibrium selection and cascaded failures.
\newblock {\em IEEE Transactions on Automatic Control}, 58(2):333--348, 2013.

\bibitem{coogan2015compartmental}
Samuel Coogan and Murat Arcak.
\newblock A compartmental model for traffic networks and its dynamical
  behavior.
\newblock {\em IEEE Transactions on Automatic Control}, 60(10):2698--2703,
  2015.

\bibitem{coogan2016stability}
Samuel Coogan and Murat Arcak.
\newblock Stability of traffic flow networks with a polytree topology.
\newblock {\em Automatica}, 66:246--253, 2016.

\bibitem{daganzo1994cell}
Carlos~F. Daganzo.
\newblock The cell transmission model: A dynamic representation of highway
  traffic consistent with the hydrodynamic theory.
\newblock {\em Transportation Research Part B: Methodological}, 28(4):269--287,
  1994.

\bibitem{daganzo1995cell}
Carlos~F. Daganzo.
\newblock The cell transmission model, part ii: network traffic.
\newblock {\em Transportation Research Part B: Methodological}, 29(2):79--93,
  1995.

\bibitem{gomes2006optimal}
Gabriel Gomes and Roberto Horowitz.
\newblock Optimal freeway ramp metering using the asymmetric cell transmission
  model.
\newblock {\em Transportation Research Part C: Emerging Technologies},
  14(4):244--262, 2006.

\bibitem{gomes2008behavior}
Gabriel Gomes, Roberto Horowitz, Alex~A Kurzhanskiy, Pravin Varaiya, and
  Jaimyoung Kwon.
\newblock Behavior of the cell transmission model and effectiveness of ramp
  metering.
\newblock {\em Transportation Research Part C: Emerging Technologies},
  16(4):485--513, 2008.

\bibitem{gurobi}
Inc. Gurobi~Optimization.
\newblock Gurobi optimizer reference manual, 2016.

\bibitem{hirsch2005monotone}
MW~Hirsch and H~Smith.
\newblock Monotone maps: A review.
\newblock {\em Journal of Difference Equations and Applications},
  11(4-5):379--398, 2005.

\bibitem{kontorinaki2017first}
Maria Kontorinaki, Anastasia Spiliopoulou, Claudio Roncoli, and Markos
  Papageorgiou.
\newblock First-order traffic flow models incorporating capacity drop: Overview
  and real-data validation.
\newblock {\em Transportation Research Part B: Methodological}, 106:52--75,
  2017.

\bibitem{kurzhanskiy2010active}
Alex~A Kurzhanskiy and Pravin Varaiya.
\newblock Active traffic management on road networks: a macroscopic approach.
\newblock {\em Philosophical Transactions of the Royal Society of London A:
  Mathematical, Physical and Engineering Sciences}, 368(1928):4607--4626, 2010.

\bibitem{kurzhanskiy2012guaranteed}
Alex~A Kurzhanskiy and Pravin Varaiya.
\newblock Guaranteed prediction and estimation of the state of a road network.
\newblock {\em Transportation Research Part C: Emerging Technologies},
  21(1):163--180, 2012.

\bibitem{lighthill1955kinematic}
Michael~J. Lighthill and Gerald~Beresford Whitham.
\newblock On kinematic waves ii. a theory of traffic flow on long crowded
  roads.
\newblock {\em Proceedings of the Royal Society of London A: Mathematical,
  Physical and Engineering Sciences}, 229(1178):317--345, 1955.

\bibitem{lovisari2014stability}
Enrico Lovisari, Giacomo Como, Anders Rantzer, and Ketan Savla.
\newblock Stability analysis and control synthesis for dynamical transportation
  networks.
\newblock {\em arXiv preprint arXiv:1410.5956}, 2014.

\bibitem{lovisari2014stability2}
Enrico Lovisari, Giacomo Como, and Ketan Savla.
\newblock Stability of monotone dynamical flow networks.
\newblock In {\em IEEE Conference on Decision and Control (CDC)}, pages
  2384--2389, 2014.

\bibitem{muralidharan2012optimal}
Ajith Muralidharan and Roberto Horowitz.
\newblock Optimal control of freeway networks based on the link node cell
  transmission model.
\newblock In {\em American Control Conference (ACC)}, pages 5769--5774. IEEE,
  2012.

\bibitem{papageorgiou2003review}
Markos Papageorgiou, Christina Diakaki, Vaya Dinopoulou, Apostolos Kotsialos,
  and Yibing Wang.
\newblock Review of road traffic control strategies.
\newblock {\em Proceedings of the IEEE}, 91(12):2043--2067, 2003.

\bibitem{rantzer2014control}
Anders Rantzer and Bo~Bernhardsson.
\newblock Control of convex-monotone systems.
\newblock In {\em IEEE Conference on Decision and Control (CDC)}, pages
  2378--2383, 2014.

\bibitem{richards1956shock}
Paul~I Richards.
\newblock Shock waves on the highway.
\newblock {\em Operations Research}, 4(1):42--51, 1956.

\bibitem{schmitt2017convex}
Marius Schmitt, Chithrupa Ramesh, Paul Goulart, and John Lygeros.
\newblock Convex, monotone systems are optimally operated at steady-state.
\newblock In {\em IEEE American Control Conference (ACC)}, pages 2662--2667,
  2017.

\bibitem{schmitt2017sufficient}
Marius Schmitt, Chithrupa Ramesh, and John Lygeros.
\newblock Sufficient optimality conditions for distributed, non-predictive ramp
  metering in the monotonic cell transmission model.
\newblock {\em Transportation Research Part B: Methodological}, 105:401 -- 422,
  2017.

\bibitem{varaiya2013max}
Pravin Varaiya.
\newblock The max-pressure controller for arbitrary networks of signalized
  intersections.
\newblock In {\em Advances in Dynamic Network Modeling in Complex
  Transportation Systems}, pages 27--66. Springer, 2013.

\bibitem{wachter2006implementation}
Andreas W{\"a}chter and Lorenz~T Biegler.
\newblock On the implementation of an interior-point filter line-search
  algorithm for large-scale nonlinear programming.
\newblock {\em Mathematical Programming}, 106(1):25--57, 2006.

\bibitem{ziliaskopoulos2000linear}
Athanasios~K. Ziliaskopoulos.
\newblock A linear programming model for the single destination system optimum
  dynamic traffic assignment problem.
\newblock {\em Transportation Science}, 34(1):37--49, 2000.

\end{thebibliography}

%% APPENDIX %%%%%%%%%%%%%%%%%%%%%%%%%%%%%%%%%%%%%%%%%%%%%%%%%%%%%%%%%%
\FloatBarrier

\begin{appendix}

\section{Proof of Lemma \ref{lemma:invariance}} %%% =======
\label{appendix:lemma_invariance}

\begin{proof}
For all states $\rho(t) \in \mathbb{P}$, demand and supply functions are non-negative according to Assumption \ref{assumption:fd}. This implies that the candidate input $\phi_e(t) = 0$ for all $e \in \N_S$ is always feasible in both \eqref{eq:ctm_constraints_demand} and \eqref{eq:ctm_constraints_supply}. 

We first consider networks without asymmetric junctions. In such networks, all flows are non-negative $\phi_e(t) \geq 0$ for all $e$ according to \eqref{eq:ctm_flow_a}, \eqref{eq:ctm_flow_b} and \eqref{eq:ctm_constraints_demand}. In turn, this implies that for all states $\rho(t) \in \mathbb{P}$ and all feasible inputs,
\begin{align*}
\rho_e(t+1) \geq \rho_e(t) - \frac{\Delta t}{l_e} \phi_e(t) \geq \rho_e(t) - \frac{\Delta t}{l_e} \de \geq \rho_e(t) + \frac{\Delta t}{l_e} \cdot \frac{l_e}{\Delta t} \rho_e(t) = 0 .
\end{align*}
Here, the second inequality follows from constraints \eqref{eq:ctm_flow_a}, \eqref{eq:ctm_flow_b} or \eqref{eq:ctm_constraints_demand}, depending on the cell. The third inequality follows from Lipschitz continuity according to Assumption \ref{assumption:fd}. Similarly,
\begin{align*}
\rho_e(t+1) \leq \rho_e(t) + \frac{\Delta t}{l_e} \cdot \sum_{i \in \E} \beta_{e,i} \phi_i(t) \leq \rho_e(t) + \frac{\Delta t}{l_e} \cdot \se \leq \rho_e(t) + \frac{\Delta t}{l_e} \cdot \frac{l_e}{\Delta t} \big( \bar\rho_e - \rho_e(t) \big) = \bar\rho_e
\end{align*}
for all $e \notin \mathcal{S}$ and $\tau_e \notin \M_U$. Here, the second inequality follows from \eqref{eq:ctm_flow_a}, \eqref{eq:ctm_constraints_supply} and the fact that $w_e(t) = 0$ for all $e \notin \mathcal{S}$. The third inequality follows from Lipschitz continuity according to Assumption \ref{assumption:fd}. Source cells $e \in \mathcal{S}$ and cells immediately downstream of sub-critical junctions $e: \tau_e \in \M_U$ have infinite capacity $\bar\rho_e = + \infty$ according to Assumption \ref{assumption:merges}, which completes the proof for all networks without onramp junctions. 

For networks with asymmetric junctions, we have the additional a priori assumption that states $\rho(t)$ with $s_i \big( \rho_i(t) \big) < \beta_{i,j} \cdot d_j \big( \rho_j(t) \big)$ for some $i \in \E$ and $j \in \N_A$ are unreachable. Hence,
\begin{align*}
\phi_e(t) \geq \min \Bigg\{ 0,  ~\underset{i \in \E^+(e)}{\min} \bigg\{ \frac{ s_i \big( \rho_i(t) \big) - \sum_{j \in \N_A} \beta_{i,j} \cdot \phi_j(t) }{ \beta_{i,e}  } \bigg\}  \Bigg\} \geq 0
\end{align*}
also holds for all $e \in \N_M$. This implies $\rho_e(t+1) \geq 0$ for all $e \in \E$ by the same reasoning as for networks without asymmetric junctions. Furthermore, 
\begin{align*}
\rho_e(t+1) \leq \rho_e(t) + \frac{\Delta t}{l_e} \cdot \bigg( \beta_{i,e} \cdot \frac{s_i \big( \rho_i(t) \big) - \beta_{i,j} \cdot \phi_j(t)}{\beta_{i,e}} + \beta_{i,j} \cdot \phi_{j,i}(t) \bigg) \leq \rho_e(t) + \frac{\Delta t}{l_e} \cdot \frac{l_e}{\Delta t} \big( \bar\rho_e - \rho_e(t) \big) = \bar\rho_e 
\end{align*}
for all cells $e \in \N_M$, which completes the proof.
\end{proof}

\section{Monotonicity proofs} %%% =======
This part of the Appendix provides proofs of state-monotonicity and concavity of various functions.

\subsection{Proof of Lemma \ref{lemma:demand}}
\label{appendix:lemma_demand}

\begin{proof} % begin PROOF
We first verify concavity. The density $\rho_e(t) = \rho_e(0) + \frac{1}{l_e} \cdot \left( \sum_{i \in \E} \beta_{e,i} \Phi_i(t) ~ - \Phi_e(t) + W_e(t) \right) $ is an affine function of $\Phi(t)$. The demand $\de $ is concave in $\rho_e(t)$ by Assumption \ref{assumption:fd} and hence, it is concave in $\Phi(t)$. The cumulative demand $D_e \big( \Phi(t) \big) = \Phi_e(t) + \Delta t \cdot \de$ is the non-negative sum of two concave functions and therefore, it is concave in $\Phi(t)$ \cite{boyd2004convex}. Next, we verify state monotonicity, by resorting to the basic definition of monotonicity.  For ease of notation, we will drop all time indices, i.e., write $\Phi_e$ instead of $\Phi_e(t)$.  In the following, assume $\Delta \Phi \geq 0$. We find
\begin{align*}
D_e \big(\Phi + \Delta \Phi \big) - D_e \big( \Phi \big) &= \Delta \Phi_e + \Delta t \cdot d_e \left( \rho_e(0) + \frac{\sum_{i \in \E} \beta_{e,i} ( \Phi_i + \Delta \Phi_i ) ~ - \Phi_e - \Delta \Phi_e + W_e}{l_e} \right) \\
 & \quad\quad - \Delta t \cdot d_e \left( \rho_e(0) + \frac{\sum_{i \in \E} \beta_{e,i} \Phi_i ~ - \Phi_e + W_e}{l_e} \right) \\
  &\geq \Delta \Phi_e + \Delta t \cdot d_e \left( \rho_e(0) + \frac{\sum_{i \in \E} \beta_{e,i} \Phi_i ~ - \Phi_e - \Delta \Phi_e + W_e}{l_e} \right) \\
 & \quad\quad - \Delta t \cdot d_e \left( \rho_e(0) + \frac{\sum_{i \in \E} \beta_{e,i} \Phi_{i} ~ - \Phi_e + W_e}{l_e} \right) \\
 & \geq \Delta \Phi_e - \Delta t \cdot \frac{\gamma_d}{l_e} \cdot \Delta \Phi_e \geq 0 ~,
\end{align*}
which proves monotonicity.
\end{proof} % end PROOF

\subsection{Proof of Lemma \ref{lemma:supply}} %%% =======
\label{appendix:lemma_supply}

\begin{proof} %begin PROOF
Again, we first verify concavity. The density
\begin{align*}
\rho_i(t) = \rho_i(0) + \frac{1}{l_i} \cdot \bigg( \beta_{i,e} \Phi_e(t) + \sum_{j \in \N_A} \beta_{i,j} \Phi_{j}(t) ~ - \Phi_i(t) + W_i(t) \bigg) 
\end{align*} 
is an affine function of $\Phi(t)$. The supply function $s_i ( \cdot )$ is concave by Assumption \ref{assumption:fd} and hence, the term $s_{i,e} \big( \Phi(t) \big)$ is concave in $\Phi(t)$. The cumulative supply
\begin{align*}
S_{i,e} \big( \Phi(t) , \varphi(t) \big) &:= \Phi_e(t) + \frac{\Delta t}{\beta_{i,e}} \cdot \bigg( s_{i,e} \big( \Phi(t) \big) - \sum_{j \in \N_A} \beta_{i,j} \cdot \frac{ \varphi_j(t) - \Phi_j(t)}{\Delta t} \bigg)
\end{align*}
is given as the sum of concave functions and therefore, it is jointly concave in $\Phi(t)$ and $\varphi(t)$. To verify state monotonicity, we apply the basic definition of monotonicity in a similar manner as for the cumulative demand. In the following, we again drop the time indices for ease of notation and we assume $\Delta \Phi \geq 0$. For any cell $e \notin \N$ and a downstream cell $i$ such that $\beta_{i,e} \neq 0$ we find
\begin{align*}
S_{i,e} \big( \Phi + \Delta \Phi, \varphi \big) &- S_{i,e} \big( \Phi, \varphi \big) \\
% 2nd equation
 &= \Delta \Phi_e + \frac{\Delta t}{\beta_{i,e}} \cdot s_i \left(  
 \rho_i(0) + \frac{1}{l_i} \cdot \left( \beta_{i,e} \big( \Phi_e + \Delta \Phi_e \big) + \sum_{j \in \N_A} \beta_{i,j} \big( \Phi_j + \Delta \Phi_j \big)  ~ - \big( \Phi_i  + \Delta \Phi_i \big) + W_i \right) 
 \right) \\
 & \quad\quad - \frac{\Delta t}{\beta_{i,e}} \cdot s_i \left(  \rho_i(0) + \frac{1}{l_i} \cdot \left( \beta_{i,e} \Phi_e + \sum_{j \in \N_A} \beta_{i,j} \Phi_j ~ - \Phi_i + W_i \right) \right) + \frac{1}{\beta_{i,e}} \cdot \sum_{j \in \N_A} \beta_{i,j} \cdot \Delta \Phi_j(t) \\
% 3rd equation
%  &\geq \Delta \Phi_e + \frac{\Delta t}{\beta_{e,i}} \cdot s_i \left(  
% \rho_i(0) + \frac{1}{l_i} \cdot \left( \beta_{e,i} \big( \Phi_e + \Delta \Phi_e \big) - \sum_{j \in \N_A} \beta_{j,i} \hat\Phi_j ~ -\Phi_i + W_i \right) 
% \right) \\
% & \quad\quad - \frac{\Delta t}{\beta_{e,i}} \cdot s_i \left(  \rho_i(0) + \frac{1}{l_i} \cdot \left( \beta_{e,i} \Phi_e - \sum_{j \in \N_A} \beta_{j,i} \hat\Phi_j ~ - \Phi_i + W_i \right) \right) \\
% 4th equation
  &\geq \Delta \Phi_e - \frac{\Delta t}{\beta_{i,e}} \cdot \gamma_s \cdot \frac{ \beta_{i,e} \Delta \Phi_e + \sum_{j \in \N_A} \beta_{i,j} \Delta \Phi_j }{l_i}  + \frac{1}{\beta_{i,e}} \cdot \sum_{j \in \N_A} \beta_{i,j} \cdot \Delta \Phi_j(t) ~ \geq 0 ~,
\end{align*}
which proves state monotonicity.
\end{proof} % end PROOF

\subsection{State monotonicity of constraint \eqref{eq:csm_constraints_supply}} %%% =======
\label{appendix:proof_constraint_supply}

To verify state monotonicity, we drop the time indices for ease of notation and we assume $\Delta \Phi \geq 0$. For all $ e : \tau_e \in \M_S$, it follows analogously to the reasoning in \ref{appendix:lemma_supply} that
\begin{align*}
& \Delta t \cdot s_e \Bigg(    \rho_e(0) + \frac{1}{l_e} \bigg( \sum_{i \in \E} \beta_{e,i} \big( \Phi_i + \Delta \Phi_i  \big) - \Phi_e - \Delta \Phi_e + W_e \bigg) \Bigg) + \sum_{i \in \E} \beta_{e,i} \cdot \big( \Phi_i + \Delta \Phi_i - \varphi_i \big) \\
&\quad\quad - \Delta t \cdot s_e \Bigg(    \rho_e(0) + \frac{1}{l_e} \bigg( \sum_{i \in \E} \beta_{e,i} \Phi_i - \Phi_e + W_e \bigg) \Bigg) + \sum_{i \in \E} \beta_{e,i} \cdot \big( \Phi_i - \varphi_i \big) \\
& \geq  \Delta t \cdot s_e \Bigg(    \rho_e(0) + \frac{1}{l_e} \bigg( \sum_{i \in \E} \beta_{e,i} \big( \Phi_i + \Delta \Phi_i \big) - \Phi_e + W_e \bigg) \Bigg) + \sum_{i \in \E} \beta_{e,i} \cdot \Delta \Phi_i \\
&\quad\quad - \Delta t \cdot s_e \Bigg(    \rho_e(0) + \frac{1}{l_e} \bigg( \sum_{i \in \E} \beta_{e,i} \Phi_i - \Phi_e + W_e \bigg) \Bigg) \\
& \geq \bigg( \sum_{i \in \E} \beta_{e,i} \Delta \Phi_i \bigg) - \Delta t \cdot \frac{\gamma_s}{l_e} \cdot \bigg( \sum_{i \in \E} \beta_{e,i} \Delta \Phi_i \bigg) ~ \geq 0 ~,
\end{align*}
which proves state monotonicity.

\section{Equivalence of the relaxed problems} %%% ======= Derivation of Constraints
\label{appendix:backtransformation}

The aim of this section is to verify that the relaxed FNC problem \eqref{eq:problem_relaxed} is equivalent to the relaxed FNC in cumulative states \eqref{eq:problem_cctm_relaxed} in the sense that the former problem can be converted into the latter one by the transformation defined by the equations defining cumulative flows \eqref{eq:cumulative_flow} and cumulative inputs and vice versa by the transformation \eqref{eq:state_transformation}.

First, we have seen in Section \ref{sec:CCTM} that the objectives $\TTS = \sum_{t=1}^{T} \sum_{e \in \E} l_e \cdot \rho_e(t) = C_{\text{W}} - \sum_{t=1}^{T} \sum_{e \in \E} \hat c_e \Phi_e(t)$ are equivalent under the suggested transformation. Second, we show that the constraints are equivalent under the transformation as well. The relaxed system equations for the controlled, cumulative flows \eqref{eq:csm_flow_c} and \eqref{eq:csm_flow_d} are satisfied by definition of the cumulative flows and inputs and in the reverse direction, the conservations law \eqref{eq:ctm_density} is satisfied because of the transformation of cumulative flows into densities \eqref{eq:density_transformation}. In the same manner as in the proof of Lemma \ref{lemma:CCTM}, one can show that the remaining equations defining the relaxed CCTM are equivalent to the remaining equations of the CTM 
\begin{subequations}
\begin{align}
&\phi_e(t) \leq \min \left\{ \de,  ~\underset{i \in \E^+(e)}{\min} \left\{ \frac{ s_i \big( \rho_i(t) \big) - \sum_{j \in \N_A} \beta_{i,j} \cdot \phi_j(t) }{ \beta_{i,e}  } \right\}  \right\} &\forall e \in \mathcal{L} ~, \label{eq:con_1a} \\
&\phi_e(t) \leq \de &\forall e \in \N_U ~, \label{eq:con_1b} \\
&0 \leq \phi_e(t) \leq \de \quad & \forall e \in \N_S \cup \N_A ~, \label{eq:con_1c} \\ 
&\sum_{i \in \E} \beta_{e,i} \cdot \phi_i(t) \leq \se & \forall e : \tau_e \in \M_S ~. \label{eq:con_1d}
\end{align}
\end{subequations}
where the flow constraints are relaxed as well. Equation \eqref{eq:con_1a} can be split into demand constraints which, together with  \eqref{eq:con_1b} and the second inequality in \eqref{eq:con_1c} can be combined into the single demand constraint \eqref{eq:con_2a} and into supply constraints. These take the form
\begin{align*}
 &\phi_e(t) \leq \underset{i \in \E^+(e)}{\min} \left\{ \frac{ s_i \big( \rho_i(t) \big) - \sum_{j \in \N_A} \beta_{i,j} \cdot \phi_j(t) }{ \beta_{i,e}  } \right\}   \quad &&\forall e \in \mathcal{L} ~ \\
\Leftrightarrow \quad & \beta_{i,e} \cdot \phi_e(t) + \sum_{j \in \N_A} \beta_{i,j} \cdot \phi_j(t) \leq s_i \big( \rho_i(t) \big) \quad && \forall e \in \E \setminus \N \text{ and } i \in \E^+(e) ~.
\end{align*}
The latter constraint takes the same form as \eqref{eq:con_1d} except that it holds for asymmetric junctions and non-merging junctions and both constraints can be combined into $\sum_{i \in \E \setminus \N_U} \beta_{e,i} \cdot \phi_i(t) \leq \se $ for all $e \in \E \setminus \mathcal{S}$. Since source cells and cells immediately downstream of sub-critical junctions have infinite capacity according to Assumption \ref{assumption:merges}, we can generalize this constraint to \eqref{eq:con_2c}. The first inequality in constraint \eqref{eq:con_1c} is left unchanged as \eqref{eq:con_2b} and we obtain the equivalent system of inequalities
\begin{subequations}
\begin{align}
\phi_e(t) &\leq \de \quad && \forall e \in \E ~,\label{eq:con_2a} \\
\sum_{i \in \E} \beta_{e,i} \cdot \phi_i(t) &\leq \se && \forall e \in \E  ~,\label{eq:con_2c} \\[-1.5ex]
\phi_e(t) &\geq 0 && \forall e \in \N_S \cup \N_A ~,\label{eq:con_2b}
\end{align}
\end{subequations}
which, together with the conservations law \eqref{eq:ctm_density}, are exactly the constraints of the relaxed FNC problem \eqref{eq:problem_relaxed}.

\section{Parameters Rocade Sud} \label{appendix:rocade} %%%

For simplicity, PWA demand functions $d_e(\rho_e(t)) = \min \big\{ v_e \cdot \rho_e(t), F_e \big\}$ and supply functions $s_e( \rho_e(t) ) = \min \big\{ F_e^s, w_e \cdot \big( \overline\rho_e - \rho_e(t) \big) \big\}$ are chosen. The free-flow velocity $v_e = 90$km/h equals the speed limit, the traffic jam density $\overline \rho_e = 250$cars/km is a standard choice and the storage space on onramps is set to $400$m as in \cite{gomes2006optimal,muralidharan2012optimal}, corresponding to a maximal queue length of $50$ cars. Cell sizes are chosen according to the real-world sensor locations such that cell lengths can be obtained from map data. Turning rates and maximal throughput are estimated from real data. The values are provided in Table \ref{tab:parameter_values}, where $e^+$ in $\beta_{e^+,e}$ means the mainline cell downstream of cell $e$. Data in the congested region are scattered, making identification of the supply function difficult. A trapezoidal fundamental diagram with $w_e = 1.05 \cdot \frac{F_e}{\overline \rho_e - \rho_e^c}$ and $F_e^s = 1.05 \cdot F_e$ provides a reasonable approximation, where $\rho_e^c := \frac{F_e}{v_e}$ can be interpreted as the critical density.  Choosing $F_e^s > F_e$ ensures that the capacity drop in the demand function (Section \ref{sec:heuristic}) comes into effect, since throughput is constraint by the demand of the bottleneck cell. With monotone demand functions, there is little difference. The maximal throughput of cell $e_{19}$ has been reduced to $4500$ cars/h, in order to reproduce  congestion patterns similar to the ones observed in reality. Further details on the real-world freeway are provided in \cite{de2015grenoble}.

\begin{table}[tp] %%% TABLE
\caption{Parameter values of the Grenoble Freeway}
\begin{center}
\begin{tabular}{rccccccccccc}
\toprule 
Cell & 1 & 2 & 3 & 4 & 5 & 6 & 7 & 8 & 9 & 10 & 11 \\ [1ex]
$l_e$ (km) & $0.5$ & $0.6$ & $0.5$ & $0.5$ & $0.7$ & $0.5$ & $0.5$ & $0.7$ & $1.3$ & $0.5$ & $0.5$ \\
$F_e$ (cars/h) & $4410$ & $5364$ & $5500$ & $4950$ & $5257$ & $4311$ & $4680$ & $4950$ & $5167$ & $4878$ & $4320$ \\ 
$\beta_{e^+,e}$ & $1$ & $1$ & $0.90$ & $1$ & $0.82$ & $1$ & $1$ & $1$ & $0.89$ & $1$ & $1$\\   
\bottomrule
\\
\toprule 
Cell & 12 & 13 & 14 & 15 & 16 & 17 & 18 & 19 & 20 & 21 \\ [1ex]
$l_e$ (km) & $0.5$ & $0.5$ & $0.5$ & $0.5$ & $0.5$ & $0.5$ & $0.5$ & $0.5$ & $0.5$ & $0.5$ \\
$F_e$ (cars/h) & $4800$ & $4644$ & $5304$ & $4923$ & $4608$ & $5120$ & $5049$ & $4500$ & $5049$ & $7574$ \\
$\beta_{e^+,e}$ & $0.90$ & $1$ & $0.84$ & $1$ & $1$ & $0.90$ & $1$ & $0.92$ & $1$ & $1$ \\
\bottomrule
\end{tabular}
\end{center}
\label{tab:parameter_values}
\end{table}%

\end{appendix}

\end{document}